\documentclass[11pt]{amsart}\usepackage{latexsym, mathrsfs, color, tikz,tikz-cd,multirow,bbm,mathtools, amsmath, amssymb, comment}
\usetikzlibrary{matrix, positioning}
\usetikzlibrary{arrows, arrows.meta}
\usetikzlibrary{decorations.markings}
\tikzset{->-/.style={decoration={  markings,  mark=at position #1 with
			{\arrow{>}}},postaction={decorate}}}
\tikzset{-<-/.style={decoration={  markings,  mark=at position #1 with
			{\arrow{<}}},postaction={decorate}}}
\usepackage{graphics}
\usepackage{subcaption}
\usepackage{xparse}
\input{xy}
\xyoption{all}

\usepackage[colorlinks=true, pdfstartview=FitV,
 linkcolor=blue,anchorcolor=Periwinkle, citecolor=red,urlcolor=Emerald]{hyperref}

\tikzcdset{scale cd/.style={every label/.append style={scale=#1},
    cells={nodes={scale=#1}}}}

\newenvironment{red}{\relax\color{red}}{\relax}
\newenvironment{blue}{\relax\color{blue}}{\hspace*{.5ex}\relax}

\newcommand{\ber}{\begin{red}}
\newcommand{\er}{\end{red}}
\newcommand{\beb}{\begin{blue}}
\newcommand{\eb}{\end{blue}}

\theoremstyle{plain}
\newtheorem{theorem}{Theorem}[section]
\newtheorem{thmx}{Theorem}

\newtheorem{lemma}[theorem]{Lemma}
\newtheorem{corollary}[theorem]{Corollary}
\newtheorem{proposition}[theorem]{Proposition}
\theoremstyle{definition}
\newtheorem{definition}[theorem]{Definition}

\newtheorem{remark}[theorem]{Remark}
\newtheorem{notations}[theorem]{Notations}
\numberwithin{equation}{section}

\def\surf{\mathbf{S}}                       
\def\TT{\mathbf{T}}
\def\PP{\mathbf{P}}
\def\MM{\mathbf{M}}

\def\<{\langle}
\def\>{\rangle}

\renewcommand{\k}{\mathbf{k}}

\renewcommand{\mod}{\operatorname{mod}}

\newcommand{\Int}{\operatorname{Int}}
\newcommand{\Hom}{\operatorname{Hom}}
\newcommand{\Ext}{\operatorname{Ext}}
\newcommand{\udim}{\operatorname{\underline{dim}}}

\begin{document}
\title[An Intersection-Dimension Formula for preprojective modules of type $\widetilde{D}_n$]{An Intersection-Dimension Formula for preprojective modules of type $\widetilde{D}_n$}
\author[B. Jackson]{Blake Jackson}
\address{Department of Mathematics, University of Alabama,
	Tuscaloosa, AL 35487, U.S.A.}
\email{bajackson9@crimson.ua.edu}


\begin{abstract}
    This paper proves the existence of an intersection-dimension formula for preprojective modules over path algebras of type $\widetilde{D}_n$.
    Identical intersection-dimension formulas have previously been provided for modules over path algebras of type $A_n, D_n,$ and $\widetilde{A}_n$ due to Schiffler as well as He, Zhou, and Zhu.
    These modules can be represented geometrically by some set of curves on special surfaces.
    The intersection-dimension formula is an equality of the intersection number between two curves and the dimensions of the first extension spaces between the two modules they represent.
    This paper takes a direct approach to proving the formula utilizing the known structure of the Auslander-Reiten quiver of type $\widetilde{D}_n$.
    Future work will extend the formula to the entire module category (not just the preprojective modules) over path algebras of type $\widetilde{D}_n$.
\end{abstract}

\maketitle

\section{Introduction}

Cluster categories were introduced simultaneously in \cite{buan_tilting_2006,caldero_triangulated_2006} in order to further study the cluster algebras introduced by Fomin and Zelevinsky \cite{fomin_cluster_2002,fomin_cluster_2003}.
Cluster categories are a modified version of the module category of a hereditary algebra (in fact, they are an orbit category of the derived category of the module category).
In \cite{caldero_triangulated_2006}, the authors realize the cluster categories of type $A_n$ by constructing a category of diagonals of a regular polygon with $n+3$ vertices.
The construction of the Auslander-Reiten quiver from a triangulation of some special surfaces was introduced in \cite{caldero_quivers_2006,schiffler_geometric_2008}.
Schiffler \cite{schiffler_geometric_2008} provided a model for cluster (module) categories of type $D_n$ by using triangulations of a once-punctured $n$-gon.
Recently, He, Zhou, and Zhu \cite{he_geometric_2023} provided a geometric model for the module category of skew-gentle or clannish algebras (a large class of representation-tame finite-dimensional algebras introduced in \cite{geis_auslander-reiten_1999} which can be realized from partially triangulated surfaces). 
The skew-gentle algebras cover the classical path algebras of type $A_n$ as well as path algebras of type $\widetilde{A}_n$, but not $D_n$ or $\widetilde{D}_n$.

The main feature of these geometric models in regard to the module category is the so-called intersection-dimension formula for the modules over these algebras.
Modules are represented in the geometric model by some set of admissible tagged edges in their respective surface. 
The intersection-dimension formula relates the geometric data (intersections of edges) with the homological data of the modules they represent (dimension of the first extension spaces between the two modules).
The former is a straightforward calculation after the model is given; the latter can be highly nontrivial, especially for path algebras over quivers of affine Dynkin type (for example, quivers of type $\widetilde{D}_n$).

This paper continues this direction of research with an intersection-dimension formula for the preprojective modules of type $\widetilde{D}_n$, the last of the acyclic path algebras that can be realized from triangulations of a surface.
The algebraic properties of preprojective modules over acyclic path algebras are well studied (see, for example, \cite{assem_elements_2006,draxler_existence_1996,gabriel_representations_1997,ringel_tame_1984,simson_elements_2007-1,simson_elements_2007}).
However, an intersection-dimension formula for preprojective modules over algebras of type $\widetilde{D}_n$ has yet to emerge.

Schiffler's \cite{schiffler_geometric_2008} proof of the existence of the intersection-dimension formula for modules of type $D_n$ bears some similarity to my proof for preprojective modules of type $\widetilde{D}_n$.
In the type $D_n$ case, there are finitely many isomorphism classes of indecomposable modules and the largest dimension of an extension space is 2.
Since there are finitely many homotopy classes of admissible tagged edges in a once-punctured monogon, counting the intersections between two admissible tagged edges just requires knowledge of the relative position of the endpoints. 
Unfortunately, in the type $\widetilde{D}_n$ case, there are infinitely many isomorphism classes of indecomposable preprojective tagged edges with the same endpoints.

On the other hand, He, Zhou, and Zhu's \cite{he_geometric_2023} proof relies heavily on the theory of skew-gentle algebras that were introduced in \cite{geis_auslander-reiten_1999}.
An important distinction between their work and this paper is the geometric model considered.
The geometric model used by He, Zhou, and Zhu requires the punctured, marked surface to be \textit{partially} triangulated.
Due to the work of Fomin, Shapiro, and Thurston \cite{fomin_cluster_2008}, it's not hard to see that He, Zhou, and Zhu's techniques are unable to cover the type $\widetilde{D}_n$ modules. 
Their technique does cover the intersection-dimension formulas for modules of types $A_n$ and $\widetilde{A}_n$, the latter being an addition to the list of module categories exhibiting an intersection-dimension formula.

This paper takes a direct approach to prove the existence of an intersection-dimension formula for a path algebra of type $\widetilde{D}_n$, relying on a theorem detailing the structure of the preprojective component of the Auslander-Reiten quivers of affine type.
For the benefit of the reader, detailed proofs are given along with extensive examples at the end of the paper.
Future work will extend the formula in this paper to the entire module category of type $\widetilde{D}_n$ and, hopefully, to the cluster category as well.
Beyond this, the next step to furthering this line of research is to provide geometric objects which correspond to the quivers (not even modules or clusters) of exceptional types $E_i$ and $\widetilde{E}_i$ for $i = 6,7,8$.

The main result of the paper is the following theorem:

\begin{thmx} \label{main-thm}
    Let $\surf$ be a triangulated, twice-punctured marked surface whose triangulation $\TT$ corresponds to an acyclic quiver $Q^\TT$ of type $\widetilde{D}_n$, $\k$ be an algebraically closed field, and $\k Q^\TT$ be the path algebra over $Q^\TT$. 
    Then given any two preprojective tagged edges $(\gamma_1, \kappa_1)$ and $(\gamma_2, \kappa_2)$ (not necessarily distinct), 
    $$\Int((\gamma_1, \kappa_1), (\gamma_2, \kappa_2)) = \dim_\k \Ext^1 (M_1, M_2) + \dim_\k \Ext^1 (M_2, M_1)$$ where $M_i = M(\gamma_i,\kappa_i)$ is a $\k Q^\TT$-module for $i =1,2$, $\Int$ is the intersection number between two preprojective tagged edges, and $\Ext^1(M,N)$ is the group of extensions of $N$ by $M$ viewed as a $\k$-vector space spanned by short exact sequences in $\mod \k Q^\TT$. 
\end{thmx}

The paper is organized as follows.
Section 2 defines the basic objects and notation that will be used in the remainder of the paper. 
Next, Section 3 describes the main construction of the geometric model and defines a fundamental operation on elements of the geometric model--the tagged rotation of an admissible tagged edge.
Sections 4 and 5 define the two categories which are used to prove Theorem~\ref{main-thm} along with the Auslander-Reiten translation.
Section 6 proves the equivalence of these categories.
Finally, Section 7 is dedicated to the proof of Theorem~\ref{main-thm} along with extensive examples to illustrate the result.

Throughout this paper, $\k$ is an algebraically closed field.
For any $\k$-algebra $\mathcal{A}$, we consider only finite-dimensional left $\mathcal{A}$-modules where $\mod \mathcal{A}$ is the abelian category of these modules.
Finally, this paper uses the convention that an affine Dynkin diagram of type $\widetilde{D}_n$ contains $n+1$ vertices.

\section{Quivers and Path Algebras}

\begin{definition}\label{def-quiver}
    A \textbf{quiver} $Q = (Q_0, Q_1)$ is a finite digraph without loops and directed 2-cycles where $Q_0$ is the set of vertices and $Q_1$ is the set of arrows.
    The elements of $Q_0$ are indexed by the numbers $1, 2, ... n$. 
    If $\alpha \in Q_1$ with $\alpha: i \to j$, then we say $s(\alpha) = i$ is the \textbf{source} of $\alpha$ and $t(\alpha) = j$ is the \textbf{target} of $\alpha$.
    A quiver $Q$ is \textbf{acyclic} if it contains no directed cycles of any length.
\end{definition}

\begin{definition}
    Let $i \in Q_0$ be a vertex in $Q$. 
    Then $\mu_i Q$ is the \textbf{mutation} of $Q$ at vertex $i$ and is the quiver obtained from $Q$ in the following way:
    \begin{enumerate}
        \item for each path $j \to i \to k$ in $Q$ of length 2 passing through $i$, add an arrow $j \to k$ in $\mu_i Q$
        \item reverse all arrows which begin or end at vertex $i$
        \item delete any 2-cycles that have appeared as a result of step 1.
    \end{enumerate}
    A quiver $Q'$ is \textbf{mutation-equivalent} to $Q$ if there is a finite sequence of vertices $i_1, i_2, ..., i_k$ in $Q_0$ such that $Q' = \mu_{i_1}\mu_{i_2}...\mu_{i_k}Q$.
    If $G$ is a finite or affine Dynkin diagram, then a quiver $Q$ is \textbf{of type $G$} if $Q$ is mutation equivalent to an acyclic orientation of $G$. 
\end{definition}

Quivers were introduced by Gabriel \cite{gabriel_unzerlegbare_1972} and have become fundamental objects in the study of cluster algebras developed by Fomin and Zelevinsky \cite{fomin_cluster_2002,fomin_cluster_2003}.
While this paper is less concerned with cluster theory, developments in the field of cluster algebras have led to significant contributions in other areas of mathematics, such as the representation theory of algebras.
This paper is concerned primarily with quivers of type $\widetilde{D}_n$; however, care is taken to state results in the most general sense.
It is also worth noting that mutation acts as an involution on quivers (and clusters \cite{fomin_cluster_2002}).
This means that for a quiver $Q$, $\mu_i \mu_i Q = Q$.

\begin{definition}
    Let $Q$ be a finite, connected quiver and $\k$ be an algebraically closed field.
    Then $\k Q$ is the \textbf{path algebra over $Q$} with basis given by the set of all directed paths in $Q$ and multiplication between two basis elements given by path concatenation.
    If $Q$ is acyclic then the basis of $\k Q$ will have finitely many elements; in this case, the algebra is \textbf{finite-dimensional}.
    If $G$ is a finite or affine Dynkin diagram, then $\k Q$ is \textbf{of type $G$} if $Q$ is of type $G$. 
\end{definition}

\begin{definition}
    A ring or an algebra is said to be \textbf{hereditary} if all submodules of projective modules are projective.
\end{definition}

The following is a classical result of the representation theory of finite-dimensional algebras.

\begin{theorem}
    Every finite-dimensional hereditary algebra over an algebraically closed field is Morita equivalent to a path algebra over some acyclic quiver $Q$.
\end{theorem}

Hereditary algebras are algebras with global dimension 1.
This already tells us a great deal about the homological properties of these algebras: if $\mathcal{A}$ is a hereditary algebra, then $\Ext^i_{\mathcal{A}}(M,N) = 0$ for $i > 1$ and any $M,N \in \mod \mathcal{A}$.
This restricts the study of extensions to $\Ext^1_{\mathcal{A}}$.

\section{Main Construction and Tagged Rotation}

\begin{definition}\label{def-tagged}
    A \textbf{punctured, marked surface with boundary} as defined in \cite{fomin_cluster_2008} is a triple $\surf = (S, \MM, \PP)$ where $S$ is an oriented surface with boundary, $\PP \in S \setminus \partial S$ is the set of punctures, and $\MM \in \partial S$ is the set of marked points on the boundary of $S$. 
    An \textbf{admissible tagged edge or curve} $(\gamma, \kappa)$ in a punctured, marked surface with boundary $\surf$ is a continuous map $\gamma: [0,1] \to \surf$ and a map $\kappa: \{ t \mid \gamma(t) \in \PP \} \to \{0,1\}$ such that 
    \begin{enumerate}
        \item $\gamma(0), \gamma(1) \in \PP \cup \MM$,
        \item $\gamma(t) \in \Delta^0 = \surf \setminus (\PP \cup \partial S) \text{ for } 0 < t < 1$,
        \item $\gamma$ or its self-intersection do not cut out a once-punctured monogon (Figure~\ref{fig:monogon}), and
        \item $\gamma$ is not homotopic to a boundary segment of $\surf$
    \end{enumerate}
    Let $E^\times$ be the set of admissible tagged edges $(\gamma, \kappa)$. 
    If $\gamma$ has both endpoints in $\MM$, then the domain of $\kappa$ is $\emptyset$ and for convenience we write $\kappa = \emptyset$.
\end{definition}

\begin{figure}
    \centering
    \includegraphics[scale = 0.5]{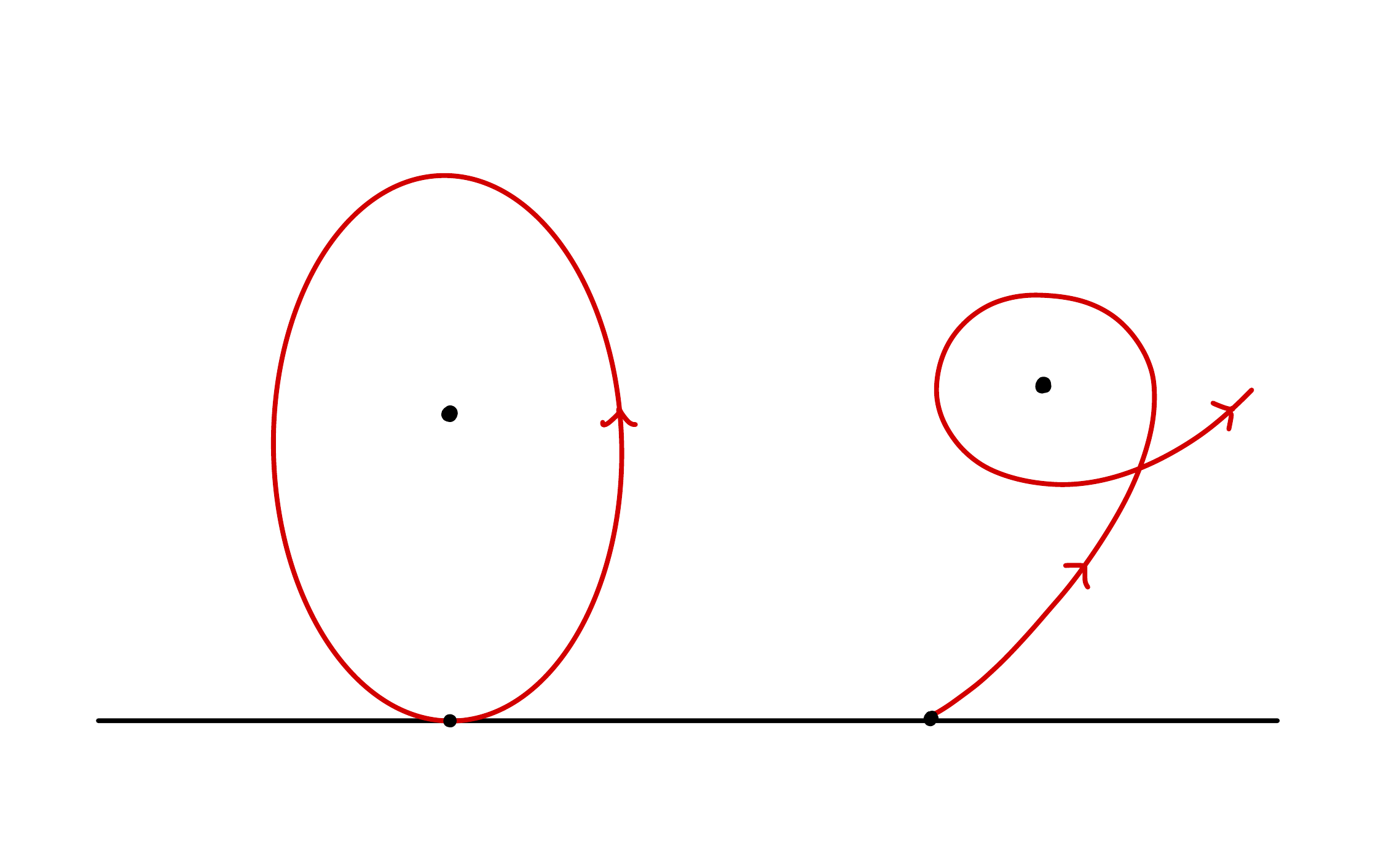}
    \caption{Non-admissible curves cutting out once-punctured monogons.}
    \label{fig:monogon}
\end{figure}

\begin{remark}
    For any admissible tagged edge with one endpoint in $\PP$ and the other in $\MM$, we will assume that $\gamma(0) \in \PP$.
    This reduces the number of choices to make later and has no effect on the results.
\end{remark}

\begin{notations}
    Graphically the map $\kappa$ is indicated by using a notch or bowtie. 
    If $\gamma(0) \in \PP$ and $\kappa(0) = 1$, then $\gamma$ will have a notch or bowtie drawn near $\gamma(0)$. 
    Otherwise, the edge will be plain near $\gamma(0)$.
    For example, in Figure~\ref{fig:badtriangles}, the edge labeled 7 in the left-hand panel and the edges labeled 1 and 2 in the right-hand panel all have $\kappa(0) = 1$.
\end{notations}

The existence of edges ending in a puncture slightly complicates the definition of intersection between two admissible tagged edges.

\begin{definition}\label{def-puctintersection}
    Two admissible tagged edges $(\gamma_1,\kappa_1), (\gamma_2,\kappa_2) \in E^\times$ are said to \textbf{intersect in a puncture} if 
    \begin{enumerate}
        \item $\gamma_1(t_1) = \gamma_2(t_2) \in \PP$
        \item $\kappa_1(t_1) \neq \kappa_2(t_2)$
        \item If $\gamma_1$ and $\gamma_2$ are homotopic as untagged edges, then $\gamma_1(1 - t_1) = \gamma_2(1 - t_2) \in \PP$ and $\kappa_1(1 - t_1) \neq \kappa_2(1 - t_2)$.
    \end{enumerate}
\end{definition}

In words, there are two cases that result in punctured intersections.
Two admissible tagged edges starting at the same puncture and ending at two different points in $\MM$ have 1 punctured intersection if they have different tags at the puncture. 
If two admissible tagged edges have their endpoints in the same two punctures with different tagging at each end, they will have 2 punctured intersections, one for each endpoint in a puncture.
However, as will be shown in Section 4, admissible tagged edges with both endpoints in the punctures are not preprojective and are therefore beyond the scope of this paper.

\begin{definition}
    Let $(\gamma_1,\kappa_1), (\gamma_2,\kappa_2) \in E^\times$ be two admissible tagged edges and $\Delta^0 = \surf \setminus (\PP \cup \partial S)$. Then the \textbf{intersection number} between these admissible tagged edges is 
    $$\Int( (\gamma_1,\kappa_1), (\gamma_2,\kappa_2) ) \coloneqq \text{min}\{ \text{Card} (\gamma_1 \cap \gamma_2 \cap \Delta^0 ) \} + \text{Card}( \mathfrak{P}( (\gamma_1,\kappa_1), (\gamma_2,\kappa_2) ) ) \geq 0$$ 
    where $ \mathfrak{P}( (\gamma_1,\kappa_1), (\gamma_2,\kappa_2) )$ counts the number of punctured intersections, Card is the cardinality of a set, and all curves $\gamma_1, \gamma_2$ are considered up to homotopy relative to their endpoints.
    Two admissible tagged edges \textbf{cross} if $\Int( (\gamma_1,\kappa_1), (\gamma_2,\kappa_2) ) \neq 0$.
    Intersections that occur in $\Delta^0$ are called \textbf{normal intersections}.
\end{definition}

\begin{definition}\label{def-tri}
    Let $\surf$ be a punctured, marked surface with boundary. 
    A \textbf{triangulation} $\TT$ of $\surf$ is a maximal collection of non-crossing admissible tagged edges in $\surf$. 
    Note that ``non-crossing'' excludes admissible tagged edges with self-intersections and punctured intersections from the triangulation.
\end{definition}

\begin{definition}\label{def-qt}
    Let $\surf$ be a punctured, marked surface with boundary and $\TT$ be a triangulation of $\surf$. 
    Then $Q^\TT$, the \textbf{quiver associated to} $\TT$, is the following quiver:
    \begin{itemize}
        \item The elements in $Q_0^\TT$ are in bijection with the tagged edges $(\gamma,\kappa) \in \TT$
        \item There is an edge $i \to j$ in $Q_1^\TT$ if and only if 
        \begin{enumerate}
            \item the tagged edges representing $i$ and $j$ in $\TT$ share a common endpoint in $\alpha_0 \in \MM \cup \PP$
            \item $j$ is the direct counter-clockwise neighbor of $i$ at $\alpha_0$
            \item $i$ and $j$ are not homotopic as untagged edges
        \end{enumerate}
    \end{itemize}
    The final step in forming $Q^\TT$ is to delete any 2-cycles; these will be generated by arcs ending at a puncture.
    If $G$ is a finite or affine Dynkin diagram, then $\TT$ is \textbf{of type $G$} if $Q^\TT$ is of type $G$. 
    A triangulation $\TT$ is said to be \textbf{acyclic} if $Q^\TT$ is. 
    This paper only considers acyclic triangulations.
\end{definition}

\begin{remark}
    Note how $Q^\TT$ in the left-hand panel of Figure~\ref{fig:badtriangles} has arrows $5 \to 6$ and $6 \to 7$. 
    This is because the untagged edges 6 and 7 are homotopic to one another, so 5 is adjacent to both of them.
    Also, in both panels of Figure~\ref{fig:badtriangles}, since only the edges labeled 6 and 7 end at the same puncture, the arrows $6 \to 7$ and $7 \to 6$ form a 2-cycle and are deleted.
\end{remark}

\begin{figure}
    \centering
    \includegraphics[scale = 0.7]{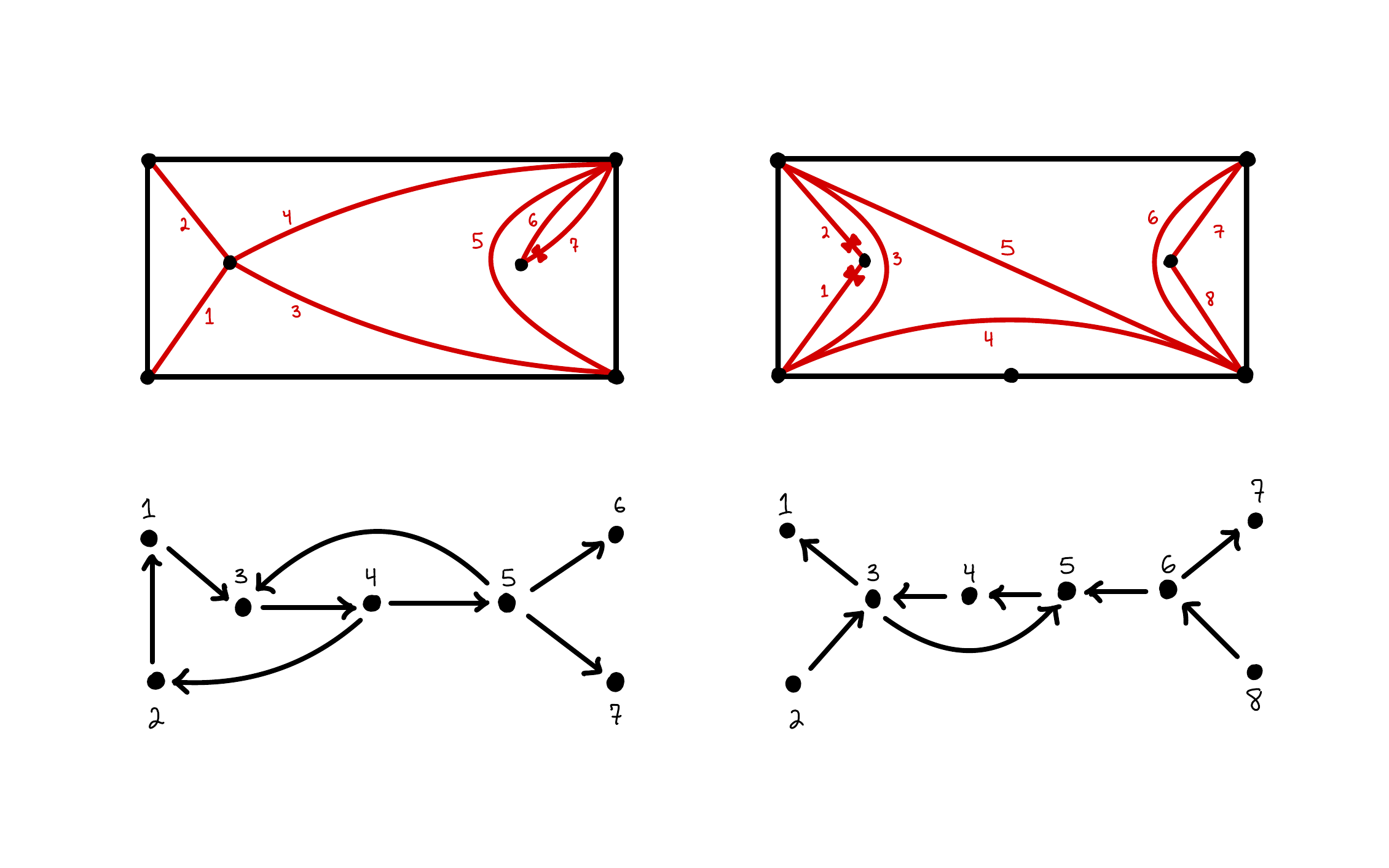}
    \caption{An example of a non-acyclic $\widetilde{D}_6$ triangulation (left) and a non-acyclic $\widetilde{D}_7$ triangulation (right).}
    \label{fig:badtriangles}
\end{figure}

The quiver associated with a triangulation $\TT$ is finite if $\TT$ is finite.
Moreover, mutation of the quiver at vertex $i$ can be realized geometrically by the so-called ``flip'' of the admissible tagged edge $i$ (see \cite{fomin_cluster_2008}). 
Figure~\ref{fig:badtriangles} gives two examples of non-acyclic triangulations of type $\widetilde{D}_n$.

\begin{definition}\label{def-rotation}
    For a curve $\gamma$ in $\surf$ with $\gamma(0) \in \MM$ (resp. $\gamma(1) \in \MM$), let $\gamma[1]$ (resp. $[1]\gamma$) be the curve obtained from $\gamma$ by moving $\gamma(0)$ (resp. $\gamma(1)$) along the boundary counterclockwise to the next marked point. 
\end{definition}

Figure~\ref{fig:endshift} gives an example of the $[1]$ operator acting on an admissible tagged edge.
The next definition is of great importance to the remainder of the paper.

\begin{definition}\label{def-rho}
    Let $(\gamma, \kappa)$ be an admissible tagged edge in $\surf$. 
    The \textbf{tagged rotation} of $(\gamma,\kappa)$ is $\rho(\gamma,\kappa) = (\rho(\gamma),\kappa')$ where 
    $$\rho(\gamma) = \begin{dcases} [1]\gamma[1] & \text{ if } \gamma(0), \gamma(1) \in \MM \\
    [1]\gamma & \text{ if } \gamma(0) \in \PP \text{ and } \gamma(1) \in \MM \\
    \gamma & \text{ if } \gamma(0), \gamma(1) \in \PP
    \end{dcases}$$ 
    and $\kappa(t)' = 1 - \kappa(t)$ for any endpoints in $\PP$.
\end{definition}

\begin{remark}
    Note that $\rho$ preserves adjacency between admissible tagged edges.
    Specifically, if $(\gamma_1,\kappa_1)$ and $(\gamma_2,\kappa_2)$ share any endpoints, so do $\rho(\gamma_1,\kappa_1)$ and $\rho(\gamma_2,\kappa_2)$.
    $\rho$ also respects self-intersections.
    As a consequence, if $\TT$ is an acyclic triangulation with quiver $Q^\TT$, then $\rho\TT$ is also an acyclic triangulation and $Q^{\rho\TT} = Q^\TT$.
    When we apply $n$ tagged rotations to a curve $(\gamma, \kappa)$, we write $\rho^n(\gamma, \kappa)$.
\end{remark}

\begin{figure}
    \centering
    \includegraphics[scale = 0.5]{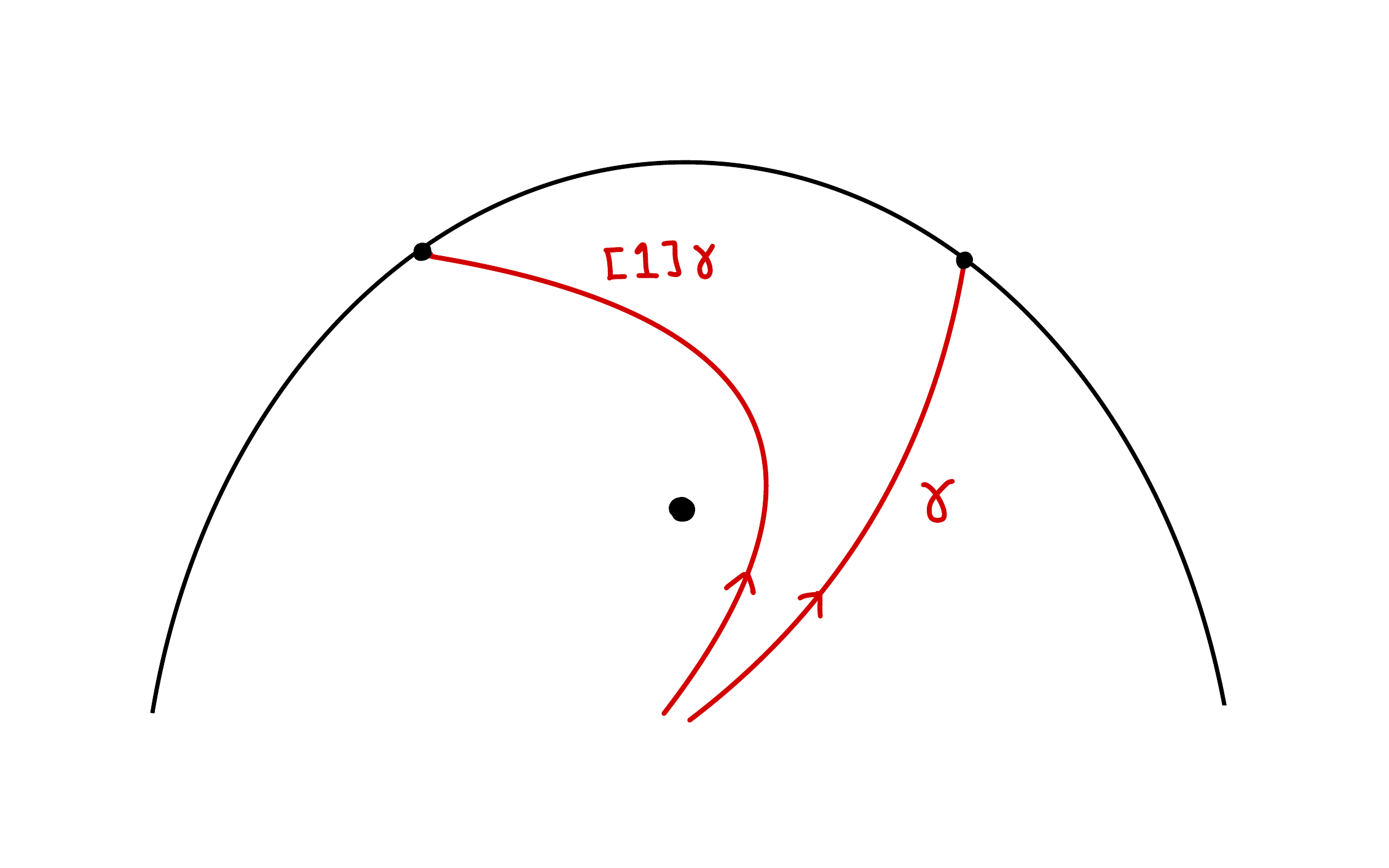}
    \caption{An example of a one-end shift.}
    \label{fig:endshift}
\end{figure}

The requirement that no admissible tagged edges cut out a once-punctured monogon is now an issue that must be considered in light of the $[1]$ operator.
In order to address this, He, Zhou, and Zhu \cite{he_geometric_2023} introduce the notion of the completion of a curve.

\begin{definition}\label{def-completion}
    If $\gamma(0) \in \PP$ and $\gamma(1) \in \MM$, then $\overline\gamma$, the \textbf{completion of $\gamma$}, is the curve which cuts out a once-punctured monogon and, together with $\gamma$, creates a self-folded triangle.
    Otherwise, $\overline\gamma = \gamma$.
    Note that there are two tagged edges ending in the puncture with the same completion; one for each map $\kappa$. 
    To avoid ambiguity, assume that all completions are oriented so that $\gamma(0) \in \PP$ lies to the left of $\overline\gamma$ (see Figure~\ref{fig:completion}).
\end{definition}

\begin{figure}
    \centering
    \includegraphics[scale = 0.5]{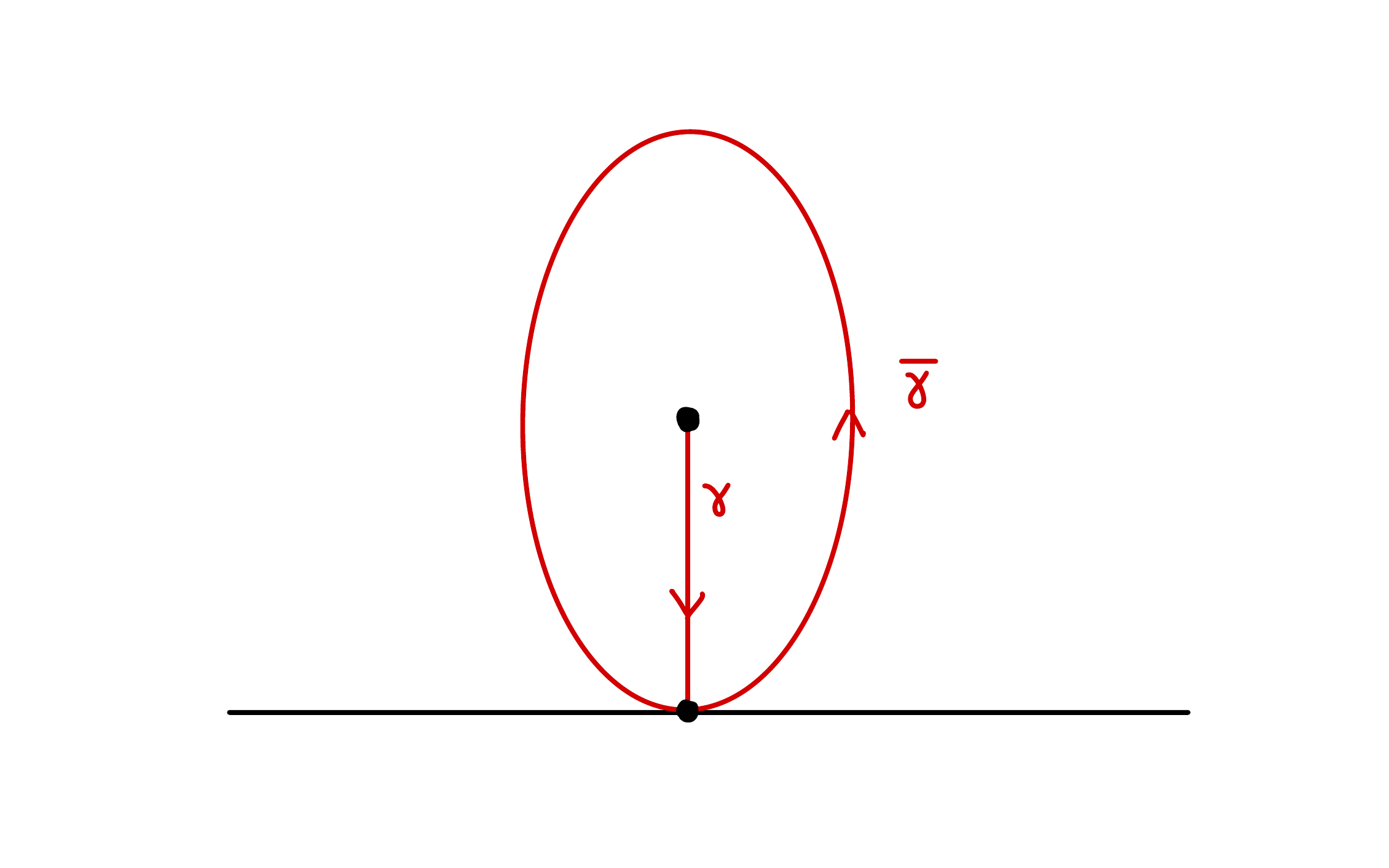}
    \caption{The completion of a tagged edge.}
    \label{fig:completion}
\end{figure}

\section{The Category of Preprojective Tagged Edges}

This is the first of the two categories that Theorem~\ref{main-thm} is concerned with.
In \cite{fomin_cluster_2008}, the authors give a precise count for the number of tagged edges in a punctured, marked surface with boundary $\surf$.
The following theorem is due to Proposition~2.10 of their paper.

\begin{theorem}[\cite{fomin_cluster_2008}] \label{thm-nplus3}
    All triangulations $\TT$ of a twice-punctured $n$-gon contain $n+3$ admissible tagged edges and have a quiver $Q^\TT$ of type $\widetilde{D}_{n+2}$.
\end{theorem}

In light of this result, this paper focuses on the case when $\surf$ is a twice-punctured $n$-gon ($n \geq 3$ so that $Q^\TT$ is at least of type $\widetilde{D}_5$).
However, smaller $n$ are sometimes considered when appropriate.
We will continue to assume that $\TT$ is an acyclic triangulation.

\begin{definition}
    Fix a triangulation $\TT$ of $\surf$.
    The \textbf{projective tagged edges} in $E^\times$ are defined to be 
    $$PTE(\TT) \coloneqq \{ (\gamma, \kappa) \in \surf \ | \ \rho^1(\gamma, \kappa) \in \TT \} \subset E^\times.$$
    The \textbf{preprojective tagged edges} in $E^\times$ are defined to be
    $$PPTE(\TT) \coloneqq \{ (\gamma, \kappa) \in \surf \ | \ \rho^n(\gamma, \kappa) \in \TT \text{ for some } n \geq 1 \} \subset E^\times.$$
\end{definition}

The following results describe $PPTE(\TT)$; specifically, they describe which classes of admissible tagged edges do not appear in $PPTE(\TT)$.

\begin{proposition} \label{prop-doublepuncturesbad}
    For a triangulation of a twice-punctured $n$-gon, $PPTE(\TT)$ contains no admissible tagged edges with both endpoints in the punctures.
    This implies, by definition, that $\TT$ contains no such admissible tagged edges.
\end{proposition}

\begin{proof}
    By definition, an admissible tagged edge $(\gamma, \kappa) \in PPTE(\TT)$ must be a finite number of tagged rotations away from an element of $\TT$. 
    If $\gamma(0), \gamma(1) \in \PP$, then the tagged rotation $\rho$ will have no effect on the endpoints of $\gamma$ and only affect the map $\kappa$.

    The proof proceeds by contradiction. 
    Assume that $(\gamma, \kappa) \in PPTE(\TT)$ has both endpoints in $\PP$.
    Since we are only considering twice-punctured $n$-gons, it must be the case that $|\PP| = 2$ and call these punctures $P_0$ and $P_1$.
    Furthermore, since admissible tagged edges (or their self-intersections) are not allowed to cut out once-punctured monogons, we can assume without loss of generality that $\gamma(0) = P_0$ and $\gamma(1) = P_1$.
    We may also assume that $\gamma$ has no self-intersections, else $\rho^n(\gamma, \kappa) \in \TT$ would have self-intersections--a contradiction to $(\gamma, \kappa) \in PPTE(\TT)$.

    Since $\rho$ sends $\kappa(t) \to \kappa'(t) = 1-\kappa(t)$ for each endpoint in a puncture, we can deduce that $\rho^2(\gamma, \kappa) = (\gamma, \kappa)$. 
    Since $(\gamma, \kappa) \in PPTE(\TT)$, it must be the case that $\rho^1(\gamma, \kappa) = (\gamma, \kappa') \in \TT$.
    Now that we have a fixed element of the triangulation that results in $(\gamma, \kappa)$ being a preprojective tagged edge, we will show that this forces the triangulation of a twice-punctured $n$-gon to be non-acyclic--a contradiction.

    \begin{figure}
        \centering
        \includegraphics[scale = 0.5]{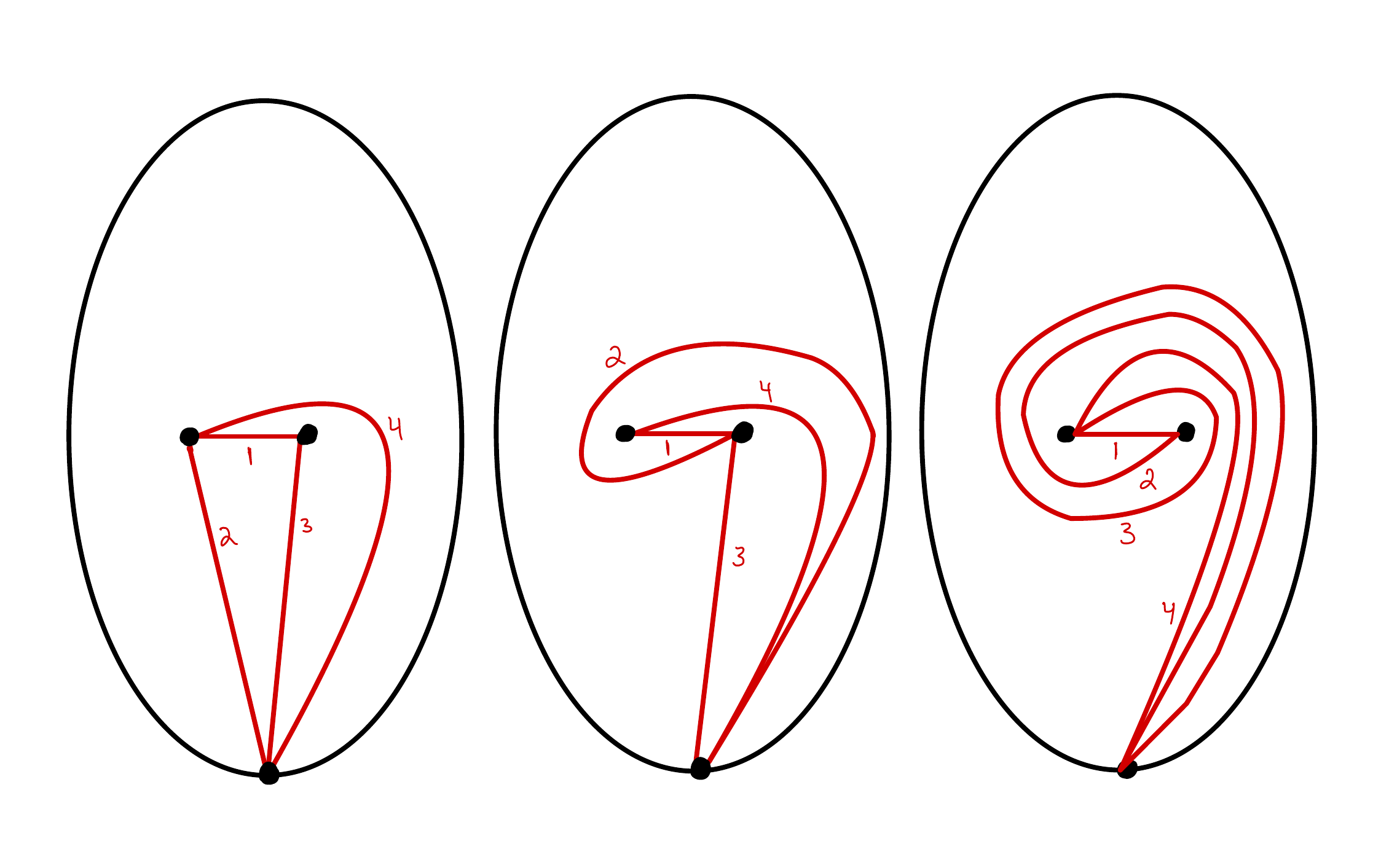}
        \caption{Triangulations of a twice-punctured monogon containing an edge with both ends in the punctures.}
        \label{fig:monogontri}
    \end{figure}

    In order to prove that all resulting triangulations will be non-acyclic, we proceed by induction.
    The base case is a twice-punctured monogon.
    It is straightforward to check that all triangulations of a twice-punctured monogon containing $(\gamma, \kappa')$ (there are three of them up to symmetry and choices for tagging as depicted in Figure~\ref{fig:monogontri}) are non-acyclic.
    The reason this happens is that, in all three cases, there will always be a puncture which is the endpoint of three or more admissible tagged edges. As can be seen in the left-hand panel of Figure~\ref{fig:badtriangles}, whenever there are three or more admissible tagged edges that end in a puncture, the resulting quiver will contain a cycle.

    Now assume $\surf$ is a twice-punctured $m$-gon containing $(\gamma, \kappa')$. 
    If there is an admissible tagged edge in the triangulation which cuts out a twice-punctured $l$-gon where $l < m$, the triangulation of the resulting $l$-gon will be non-acyclic by the inductive hypothesis and therefore the triangulation of the $m$-gon will also be non-acyclic.
    Therefore, no admissible tagged edges will have both endpoints in $\MM$, otherwise, that edge will cut out a twice-punctured $l$-gon where $l < m$.
    Hence, the remaining admissible tagged edges of the triangulation will have one endpoint in $\MM$ and the other in $\PP$. 
    There will be at least 3 admissible tagged edges ending in $\PP$: this is how many remaining admissible tagged edges there are in the twice-punctured monogon excluding $(\gamma, \kappa')$.
    Since each puncture already has $(\gamma, \kappa')$ ending in it, we can apply the pigeonhole principle to see that one of the punctures will have 3 or more admissible tagged edges ending in it.
    Therefore the resulting triangulation will be non-acyclic. 

    Since all triangulations containing $(\gamma, \kappa')$ will be non-acyclic (a contradiction), we can conclude that $PPTE(\TT)$ and $\TT$ contain no admissible tagged edges with both endpoints in the punctures.
\end{proof}

\begin{figure}
    \centering
    \includegraphics[scale = 0.5]{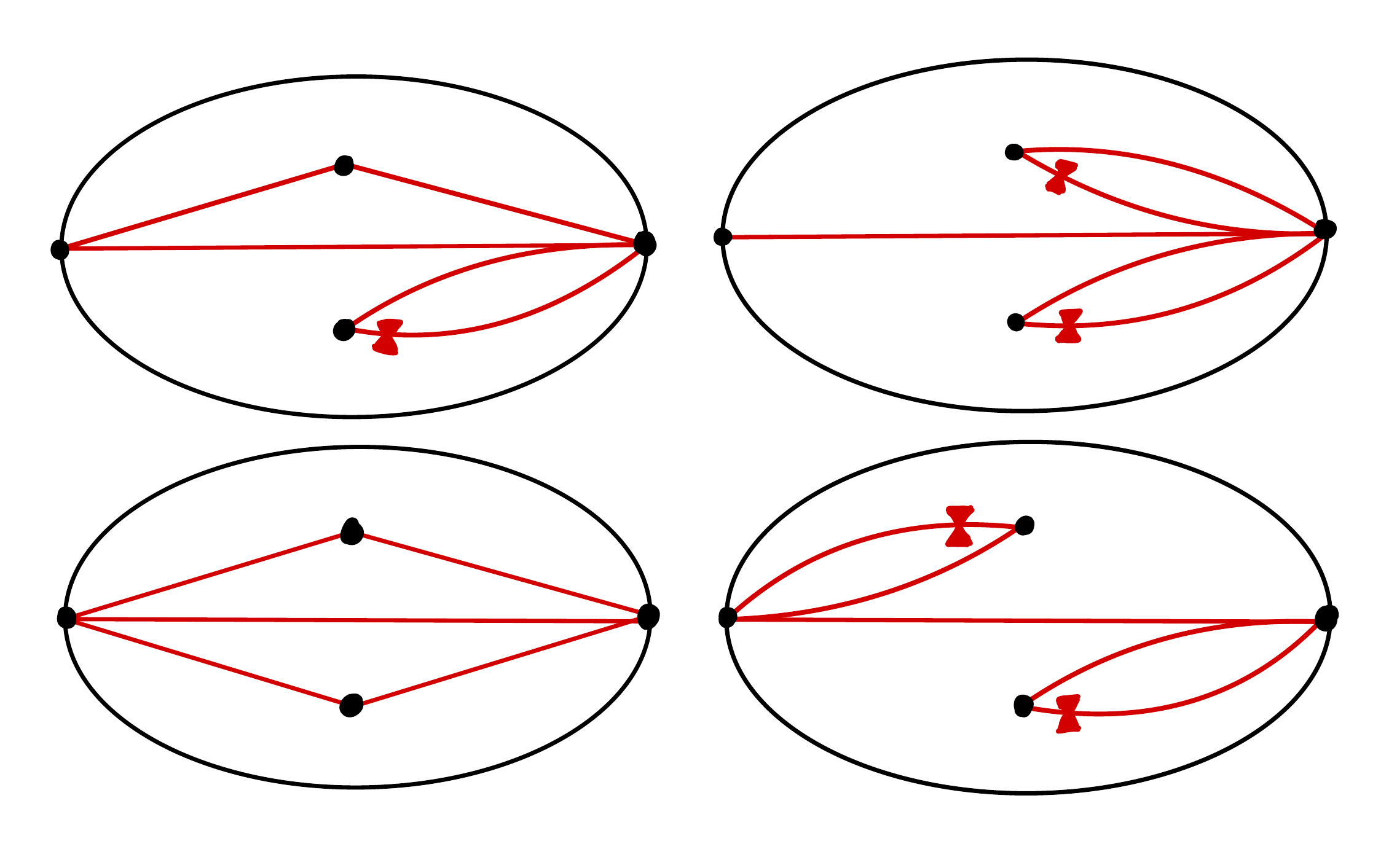}
    \caption{Acyclic triangulations of a twice-punctured digon.}
    \label{fig:digon}
\end{figure}

\begin{proposition} \label{prop-digonsbad}
    If $\surf$ is at least a twice-punctured triangle and $\TT$ is a triangulation of $\surf$, then $PPTE(\TT)$, and hence $\TT$, contains no admissible tagged edges which cut out a twice-punctured digon.
\end{proposition}
    
\begin{proof}
    Start with a twice-punctured digon and triangulate it so that the triangulation is acyclic.
    There are four of these up to symmetry and choice of tagging as depicted in Figure~\ref{fig:digon}.
    The resulting quiver will contain 5 vertices and be of type $\widetilde{D}_4$.
    Now label one of the two boundary components as an admissible tagged edge of a triangulation instead of a boundary component.
    The resulting quiver is now of type $\widetilde{D}_5$ and is non-acyclic. 
    On the other hand, if $\surf$ is at least a twice-punctured triangle and $PPTE(\TT)$ contains an admissible tagged edge that cuts out a twice-punctured digon, it must be the case that $\TT$ itself contains an admissible tagged edge which cuts out a twice-punctured digon.
    This is due to the fact that $\rho$ preserves adjacency.
    By the above argument, this triangulation is non-acyclic. 
\end{proof}

\begin{proposition} \label{prop-edgesinpunctures}
     Fix a triangulation $\TT$ of a twice-punctured $n$-gon. 
     Then each $P \in \PP$ must be the endpoint of exactly two admissible tagged edges $\gamma, \lambda \in \TT$ with the following restriction: 
     
     $\gamma(1)$ and $\lambda(1)$ must not be separated by any marked points in $\MM$.
\end{proposition}

\begin{remark}
    The condition in the statement of Proposition~\ref{prop-edgesinpunctures} means that either $\gamma$ and $\lambda$ share the same endpoint in $\MM$ or that $\gamma$ and $\lambda$ are exactly a one-end shift operation $[1]$ away from one another.
\end{remark}

\begin{proof}
    Recall that the convention in this paper is to have any admissible tagged edge with one end in $\PP$ and the other in $\MM$ to terminate in $\MM$.
    First of all, note that any acyclic triangulation has at least two admissible tagged edges ending in $P$, or the triangulation will not be a maximal set of non-crossing edges. 
    Also, by the proof of Proposition~\ref{prop-doublepuncturesbad}, any more than two admissible tagged edges ending in a single puncture will result in a non-acyclic triangulation.
    Assume that the polygon satisfies $n > 3$; if not, the proof is trivial.
    The proof proceeds by contradiction.

    Assume there is at least one marked point $m_0$ between $\gamma(1)$ and $\lambda(1)$ and that $\gamma(0) = \lambda(0) = P$ such that $m_0, \gamma(1), \lambda(1),$ and $P$ are boundary points of a polygon which is free of punctures.
    Then there is at least one other marked point $m_1$ which is separated from $m_0$ such that $m_1, \gamma(1), \lambda(1),$ and $P$ are boundary points of a once-punctured polygon (as in Figure~\ref{fig:fig-edgesinpunctures}). 
    
    In the polygon free of punctures, we must choose admissible tagged edges to progress the triangulation of the twice-punctured $n$-gon. 
    If one of these ends in the puncture $P$, this will result in 3 admissible tagged edges ending in $P$ and giving a non-acyclic triangulation--contradiction. 
    Therefore the choices for admissible tagged edges are restricted and there must be exactly one edge with endpoints $\gamma(1)$ and $\lambda(1)$.

    In the once-punctured polygon, there are more choices to consider; however, all triangulations of the once-punctured polygon will result in a non-acyclic triangulation for the original twice-punctured surface $\surf$.
    The resulting triangulation will either have an edge ending in $P$ or an edge connecting $\gamma(1)$ and $\lambda(1)$ which cuts $\surf$ into two once-punctured polygons.
    Both of these result in non-acyclic triangulations.
\end{proof}

\begin{figure}
    \centering
    \includegraphics[scale = 0.5]{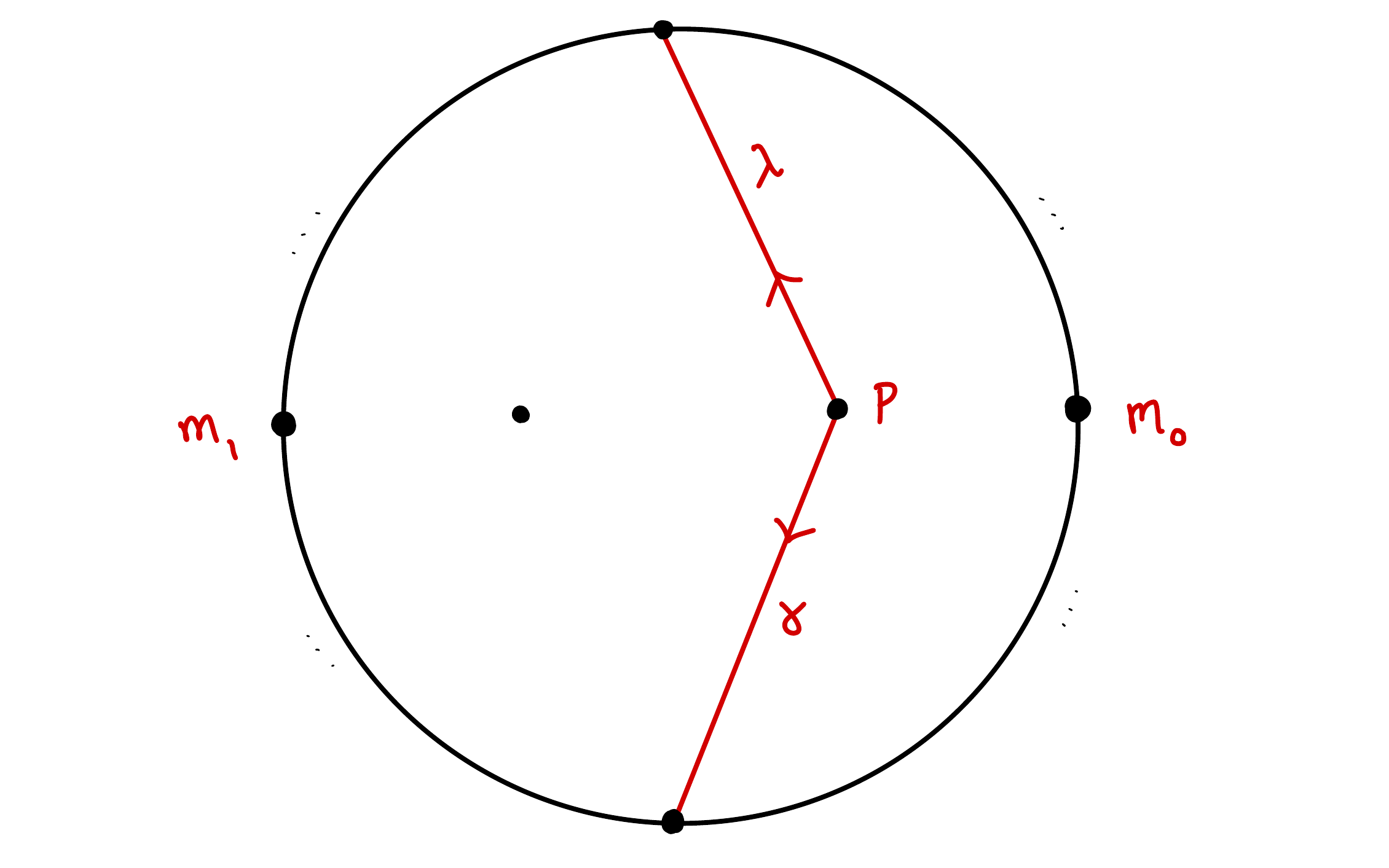}
    \caption{The set up for the proof in Proposition~\ref{prop-edgesinpunctures}.}
    \label{fig:fig-edgesinpunctures}
\end{figure}

\begin{corollary}
    If $\TT$ is a triangulation of a twice-punctured $n$-gon ($n \geq 3$) $\surf$, then $\TT$ contains two admissible tagged edges which, together with $\partial S$, cut $\surf$ into two once-punctured digons and an unpunctured polygon.
\end{corollary}

\begin{proof}
    This is a direct consequence of Proposition~\ref{prop-digonsbad} and Proposition~\ref{prop-edgesinpunctures}.
    Since the two pairs of admissible tagged edges of the triangulation that end in the two punctures must be neighbors or share the same endpoint in $\MM$, triangulations are a maximal set of non-crossing edges, and there can be no edges that cut out a twice-punctured digon, there must be two edges (one for each puncture) which cut out once-punctured digons containing the pairs.
\end{proof}

\subsection{Elementary Moves}

An elementary move sends a tagged edge $( \gamma, \kappa )$ to another tagged edge. 
Since we are only working with preprojective tagged edges, we only need to consider three classes of elementary moves. 
In other words, due to Propositions~\ref{prop-doublepuncturesbad}~and~\ref{prop-digonsbad}, there are no preprojective tagged edges with both endpoints in the punctures and no preprojective tagged edges that cut out a twice-punctured digon.

\begin{definition}
    Let $\surf$ be a twice-punctured $n$-gon, $\TT$ be a triangulation of $\surf$, and $(\gamma,\kappa) \in PPTE(\TT) \setminus PTE(\TT)$. 
    An \textbf{elementary move} is a map $PPTE(\TT) \setminus PTE(\TT) \to PPTE(\TT)$ which falls into one of three classes:
    \begin{enumerate}
        \item If $\gamma$ has both ends in $\MM$ and $\gamma$ and the boundary components of $\surf$ \textbf{do not} form a once-punctured digon, then there are exactly two elementary moves $( \gamma, \emptyset ) \mapsto ( \gamma[1], \emptyset )$ and $( \gamma, \emptyset ) \mapsto ( [1]\gamma, \emptyset )$
        \item If $\gamma$ has both ends in $\MM$ and $\gamma$ and a boundary component of $\surf$ \textbf{do} form a once-punctured digon, then there are exactly three elementary moves. Without loss of generality, assume that $\gamma$ is oriented so that the puncture inside the digon lies to the left of $\gamma$ (as in Figure~\ref{fig:elementary}). In this case, $[1]\gamma$ is homotopic to the completion $\overline{\lambda}$ of two tagged edges $(\lambda,\kappa)$ and $(\lambda,\kappa')$. The elementary moves are therefore $( \gamma, \emptyset ) \mapsto ( \gamma[1], \emptyset )$, $( \gamma, \emptyset ) \mapsto ( \lambda, \kappa)$, and $( \gamma, \emptyset ) \mapsto ( \lambda, \kappa').$
        \item If $\gamma$ satisfies $\gamma(0) \in \PP$ and $\gamma(1) \in \MM$, then there is exactly one elementary move. Let $\overline{\gamma}$ be the completion of $\gamma$. The elementary move is $( \gamma, \kappa ) \mapsto ( \overline{\gamma}[1], \emptyset ).$ 
    \end{enumerate}
\end{definition}

Note that the elementary moves described above can be thought of as possible ``middle steps'' between a curve $(\gamma,\kappa)$ and its tagged rotation $\rho(\gamma,\kappa)$.
The purpose of this will become clear in the next subsection.

\begin{figure}
    \centering
    \includegraphics[scale = 0.7]{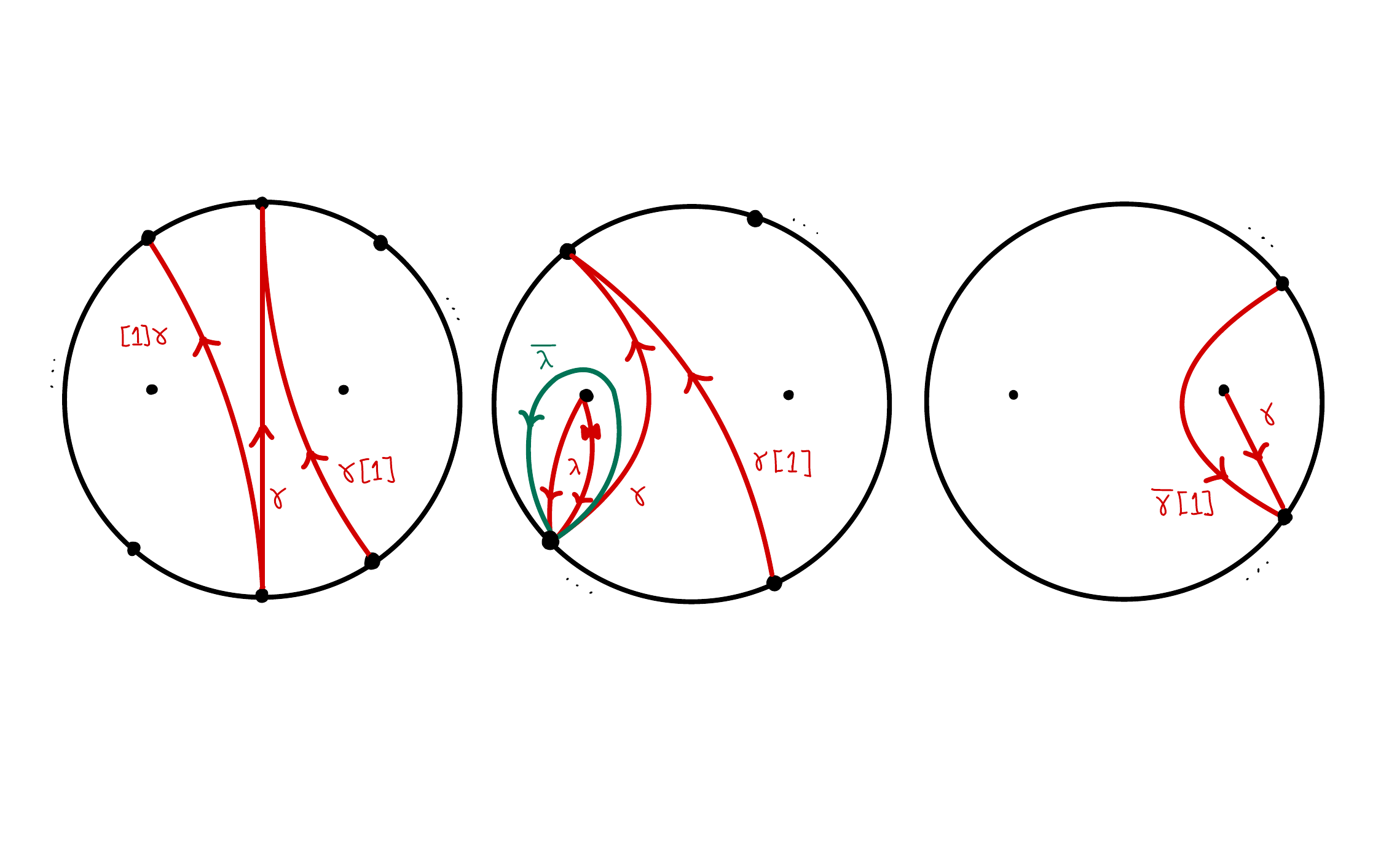}
    \caption{The three classes of elementary moves.}
    \label{fig:elementary}
\end{figure}

\subsection{The Category of Preprojective Tagged Edges}
Now that $PPTE(\TT)$ has been described and we have the notion of an elementary move, we can give the definition of the first category which plays a major role in Theorem~\ref{main-thm}.

\begin{definition}
    The \textbf{category of preprojective tagged edges} $\mathcal{P(\TT)}$ is the category with objects being homotopy classes of preprojective tagged edges. The space of morphisms from $(\gamma_1, \kappa_1)$ to $(\gamma_2, \kappa_2)$ is a quotient of the vector space spanned by elementary moves $(\gamma_2, \kappa_2) \mapsto (\gamma_1, \kappa_1)$.

    A \textbf{mesh relation} is an equality between certain sequences of elementary moves.
    To be precise, for $(\gamma, \kappa) \in PPTE(\TT)$, the mesh relation is $$m_{(\gamma,\kappa)} = \sum_\alpha (\alpha)\alpha$$
    where the sum is over all elementary moves $\alpha: (\lambda_i, \kappa_i) \mapsto (\gamma, \kappa)$ and $(\alpha)$ is the elementary move $(\alpha): \rho(\gamma, \kappa) \mapsto (\lambda_i, \kappa_i)$ which is guaranteed to exist for $(\gamma, \kappa) \notin PTE(\TT)$.
    
    Therefore, the set of morphisms from $(\gamma_1, \kappa_1)$ to $(\gamma_2, \kappa_2)$ in $\mathcal{P(\TT)}$ is the quotient of the vector space spanned by elementary moves $(\gamma_2, \kappa_2) \mapsto (\gamma_1, \kappa_1)$ by the subspace generated by the mesh relations.
\end{definition}

\begin{remark}
    It is crucial to note that the arrows in $\mathcal{P(\TT)}$ go in the \textit{opposite} direction of the elementary moves.
    Also, note that the category of preprojective tagged edges has a definite ``beginning'' starting with the projective tagged edges.
\end{remark}

\begin{proposition} \label{prop-base}
    Let $(\gamma_i, \kappa_i), (\gamma_j, \kappa_j) \in PTE(\TT)$ be two projective tagged edges corresponding to triangulated edges $i$ and $j$ respectively. 
    Then there is an arrow $(\gamma_i, \kappa_i) \to (\gamma_j, \kappa_j)$ in $\mathcal{P(\TT)}$ if and only if there is an arrow $j \to i$ in $Q^\TT$.
\end{proposition}

\begin{proof}
    Note that the requirement that $\TT$ be an acyclic triangulation ensures that there are always exactly two triangulated edges with endpoints in the same puncture which are either direct neighbors or share the same endpoint in $\MM$ (Proposition~\ref{prop-edgesinpunctures}).
    This in turn means that there will be no edges in $Q^\TT$ arising from punctured neighbors as these two edges will form a 2-cycle in $Q^\TT$.
    Furthermore, since the projective tagged edges are exactly one tagged rotation away from $\TT$, we can also say that there are exactly two projective tagged edges with endpoints in the same puncture. 
    This means that all arrows in $Q^\TT$ arise from two admissible tagged edges that are neighbors in $\MM$.
     
    If there is an arrow $j \to i$ in $Q^\TT$, then $i$ must be the direct counterclockwise neighbor of $j$ in $\TT$ with shared endpoint $\alpha_0 \in \MM$.
    Since they share one endpoint in $\MM$ and there are no triangulated edges between $i$ and $j$, it must be one of the cases depicted in Figure~\ref{fig:elementary}.
    It follows that there must be an elementary move $j \mapsto i$.
    Also, since $\rho$ preserves adjacency between two edges, we have that there is an elementary move $(\gamma_j, \kappa_j) \mapsto (\gamma_i, \kappa_i)$.
    By definition of the category of preprojective tagged edges, there will be an arrow $(\gamma_i, \kappa_i) \to (\gamma_j, \kappa_j)$ in $\mathcal{P(\TT)}$.
\end{proof}

\begin{definition} \label{def-coord}
    The \textbf{coordinates} of $(\gamma, \kappa) \in \mathcal{P(\TT)}$ are $(n,j)$ if $\rho^n(\gamma, \kappa) = j \in \TT$.
    Moreover, the \textbf{level} of $(\gamma, \kappa) \in \mathcal{P(\TT)}$ is $level(\gamma, \kappa) = n$.
\end{definition}

\begin{definition}
    Let $Q$ be an acyclic quiver. 
    Define $\mathbb{N}Q$ to be the following quiver:
    \begin{itemize}
        \item the vertex set is $\mathbb{N} \times Q_0$ with coordinates $(n,i)$ for each $n \in \mathbb{N}$ and $i \in Q_0$.
        \item for each arrow $i \to j \in Q_1$ and each $n \in \mathbb{N}$, there are two arrows $(n,i) \to (n,j)$ and $(n,j) \to (n+1,i)$ in $\mathbb{N}Q$   
    \end{itemize}
\end{definition}

\begin{lemma} \label{lem-postedge}
    $\mathcal{P(\TT)} \cong \mathbb{N} (Q^\TT)^{op}$ where $Q^{op}$ is the opposite quiver of $Q$ obtained by reversing all arrows in $Q$.
\end{lemma}

\begin{proof}
    By Definition~\ref{def-coord}, $level(\gamma,\kappa) = 1$ if and only if $(\gamma,\kappa) \in PTE(\TT)$. 
    By Proposition~\ref{prop-base}, we know that there is an arrow $(1,i) \to (1,j)$ if and only if there is an arrow $j \to i$ in $Q^\TT$. 
    Finally, by the definition of $\mathcal{P(\TT)}$ and the definition of an elementary move, we know that for each arrow $(n,i) \to (n,j)$, there will be an arrow $(n,j) \to (n+1,i)$.
    The result follows.
\end{proof}

Note that $\rho$ can be thought of as a translation in $\mathcal{P(\TT)}$. 
Specifically, if $n > 1$, then $\rho(n,i) = (n-1, i)$ and if there is an arrow $(n,i) \to (n,j)$ in $\mathcal{P(\TT)}$, then there will be an arrow $\rho(n,i) = (n-1,i) \to (n-1,j) = \rho(n,j)$.
Also, by the definitions of the elementary move and $\rho$, we can see that there will be three types of meshes in this category: one for each class of elementary move defined earlier.
These observations will be echoed in the next section.

\section{The Category of Preprojective $\widetilde{D}_n$-modules}
This section defines the second category which is featured in Theorem~\ref{main-thm}--the category of preprojective modules of type $\widetilde{D}_n$. 
Preprojective module components of algebras have been studied extensively \cite{assem_elements_2006,auslander_representation_1995,draxler_existence_1996,gabriel_representations_1997,geis_rigid_2006,geis_rigid_2007,ringel_tame_1984,simson_elements_2007}.
Due to the well-documented nature of these preprojective components, this section mainly serves as a brief overview of topics that are relevant to the present paper.

\begin{definition}
    Let $Q$ be a quiver and $\k$ be an algebraically closed field.
    A \textbf{representation} $M = (M_i, \varphi_\alpha)_{i \in Q_0, \alpha \in Q_1}$ of $Q$ is a collection of $\k$-vector spaces $M_i$ (one for each vertex in $Q_0$) and a collection of $\k$-linear maps $\varphi_\alpha: M_{s(\alpha)} \to M_{t(\alpha)}$ (one for each arrow in $Q_1$).
    The representation is \textbf{finite-dimensional} if each $M_i$ is. 
    If $M$ is finite-dimensional, the \textbf{dimension vector} $\udim M$ of $M$ is the vector $(\dim M_i)_{i \in Q_0}$ of the dimensions of the vector spaces at each vertex.
    
    If $M = (M_i, \varphi_\alpha)$ and $M' = (M'_i, \varphi'_\alpha)$ are two representations of $Q$, a \textbf{homomorphism} of representations $f: M \to M'$ is a collection $(f_i)_{i\in Q_0}$ of linear maps $f_i: M_i \to M'_i$ such that for each arrow $\alpha: i \to j$ in $Q_1$, we have commutative diagrams $f_j \circ \varphi_\alpha(m) = \varphi'_\alpha \circ f_i(m)$ for all $m\in M_i$.
    The abelian category of all finite-dimensional representations of $Q$ with morphisms given by representation homomorphisms is denoted by $\operatorname{rep} Q$.
\end{definition}

\begin{theorem}[\cite{schiffler_quiver_2014}, Theorem 5.4]\label{reps-modules}
    Let $Q$ be a finite, connected, acyclic quiver. 
    Then the finite-dimensional representations of $Q$ are in bijection with the finite-dimensional $\k Q$ modules (up to isomorphism). 
    This bijection also applies to their homomorphisms and respects the composition of these homomorphisms.
\end{theorem}

Most modern work dealing with the representation theory of hereditary algebras $\mathcal{A} = \k Q$ make no distinction between a $\k Q$-module and its corresponding quiver representation.
The next definition defines a fundamental operation in the study of $\mod \k Q$: the Auslander-Reiten translate.

\begin{definition}\label{def-AR}
    Let $\mathcal{A} = \k Q$ be an irreducible hereditary algebra and let $A, B, C, M$ and $N$ be $\mathcal{A}$-modules.
    A morphism $h$ (dually $g$) in $\mod\mathcal{A}$ is a \textbf{section} (dually \textbf{retraction}) if $h$ is a right (left) inverse of some morphism in $\mod\mathcal{A}$.
    A morphism $f: A \to B$ in $\mod\mathcal{A}$ is called \textbf{irreducible} if
    \begin{enumerate}
        \item $f$ is not a section nor a retraction
        \item whenever $f = gh$ for some morphisms $h: A \to C$ and $g: C \to B$, then either $h$ is a section or $g$ is a retraction.
    \end{enumerate}
    A short exact sequence in $\mod\mathcal{A}$ $$0 \to A \to B \to C \to 0$$ is \textbf{split} if $B \cong A \oplus C$.
    A short exact sequence in $\mod\mathcal{A}$ $$0 \to N \xrightarrow{h} B \xrightarrow{g} M \to 0$$ is \textbf{almost split} (or is an \textbf{Auslander-Reiten sequence}) if $M$ and $N$ are indecomposable and $h$ and $g$ are irreducible morphisms.
    In this case, $N$ is uniquely determined and $N \cong \tau M$ where $\tau$ is the \textbf{Auslander-Reiten translation}.
\end{definition}

\begin{remark}
    The conditions given in Definition~\ref{def-AR} for an Auslander-Reiten sequence have some equivalent formulations. 
    For example, a short exact sequence 
    $$0 \to N \xrightarrow{h} B \xrightarrow{g} M \to 0$$
    is an Auslander-Reiten sequence if 
    \begin{enumerate}
        \item the sequence is not split.
        \item $h$ is not a section; for every morphism $u: N \to A$ in $\mod\mathcal{A}$ which is not a section, there exists a morphism $u': B \to A$ such that $u' f = u$; and if $k:B \to B$ exists such that $kh = f$, then $k$ is an automorphism.
        \item $g$ is not a retraction; for every morphism $v: A \to M$ in $\mod\mathcal{A}$ which is not a retraction, there exists a morphism $v': A \to B$ such that $g v' = v$; and if $k:B \to B$ exists such that $gk = g$, then $k$ is an automorphism.
    \end{enumerate}
    However, the conditions given in Definition~\ref{def-AR} are much more manageable.
\end{remark}

\begin{definition}
    Let $P(j)$ be the indecomposable projective module at vertex $j$ and let $\tau$ be the Auslander-Reiten translation. 
    Then a $\k Q$-module $M$ is called \textbf{preprojective} if $\tau^{n-1}M \cong P(j)$ for some $j \in Q_0$; in this case, the \textbf{coordinates} of $M$ will be $(n, j)$.
\end{definition}

The next definition is for the Auslander-Reiten quiver. 
It is an important object in the study of modules and/or clusters over hereditary algebras. 

\begin{definition}
    Let $\mathcal{A} = \k Q$ be an irreducible hereditary algebra. Then the \textbf{Auslander-Reiten quiver} $\Gamma(\mod \mathcal{A} )$ is the category defined as follows:
    \begin{itemize}
        \item The vertices of $\Gamma( \mod \mathcal{A} )$ are the isomorphism classes $[M]$ of indecomposable modules $M$ in $\mod \mathcal{A}$.
        \item There is an arrow $[M] \to [N]$ in $\Gamma( \mod \mathcal{A} )$ if and only if there is an irreducible morphism $M \to N$ in $\mod \mathcal{A}$.
    \end{itemize}
    In order to keep the focus on studying the extension spaces in this category (and therefore short exact sequences of modules), we need to ``glue together'' Auslander-Reiten sequences ending at $M$ (and starting at $\tau M$).
    These glued Auslander-Reiten sequences will form the \textbf{meshes} of $\Gamma(\mod \mathcal{A})$.
    If $\mathcal{A}$ is a representation-infinite hereditary algebra, then the \textbf{preprojective component} of $\Gamma( \mod \mathcal{A} )$ is denoted by $\mathcal{P(A)}$ and is the unique connected component of $\Gamma( \mod \mathcal{A} )$ that contains all indecomposable projective $\mathcal{A}$-modules.
\end{definition}

For path algebras of acyclic type $A_n, D_n,$ and $E_i$ for $i = 6,7,8$, the Auslander-Reiten quiver is finite and connected. 
This is because these algebras are of finite representation type \cite{gabriel_representations_1997}.
In this case, all modules are preprojective.

If the path algebra is of acyclic type $\widetilde{A}_n, \widetilde{D}_n,$ or $\widetilde{E}_i$ for $i = 6,7,8$, then the situation is more complicated.
The Auslander-Reiten quiver will have 3 components: the preprojective component containing all of the projective modules, the preinjective component containing all of the injective modules, and the regular component which is a disjoint union of finitely many ``tubes''.
Each of these components has a similar local structure: the meshes.
However, the global structure of the regular component is unlike the preprojective and preinjective components while the preprojective and preinjective components mirror one another. 
The geometric model presented in this paper can easily be extended to include ``preinjective'' and ``regular'' tagged edges and will be extended in a future paper. 

\begin{remark}
    Note the similar roles that $\tau$ and $\rho$ play in their respective categories.
    If an isoclass of a module in $\mathcal{P(A)}$ has coordinates $(n,i)$ where $n>1$, then $\tau(n,i) = (n-1,i)$.
\end{remark}

There are three types of meshes that appear in an Auslander-Reiten quiver of type $\widetilde{D}_n$, these are displayed in Figure~\ref{fig:meshes}.
Note the similarity of these meshes to the meshes in $\mathcal{P}(\TT)$ which had a mesh for each class of elementary move.
The following fact cements these similarities between the two categories.

\begin{figure}
    \[\begin{tikzcd}
     & N_1 \arrow{dr} & \\
     \tau L\arrow{ur} &  & L
    \end{tikzcd}\hspace{8mm}
    \begin{tikzcd}
     \tau L\arrow{dr} &  & L \\
     & N_1 \arrow{ur} & 
    \end{tikzcd}\]
    \vspace{8mm}
    \[\begin{tikzcd}
     & N_1 \arrow{dr} & \\
     \tau L\arrow{ur}\arrow{dr} &  & L \\
      & N_2\arrow{ur} &
    \end{tikzcd}\hspace{8mm}
    \begin{tikzcd}
     & N_1 \arrow{dr} & \\
     \tau L\arrow{ur}\arrow{dr}\arrow{r} & N_2\arrow{r} & L \\
      & N_3\arrow{ur} &
    \end{tikzcd}\]  
    \caption{Meshes of the Auslander-Reiten quiver of type $\widetilde{D}_n$.}
    \label{fig:meshes}
\end{figure}
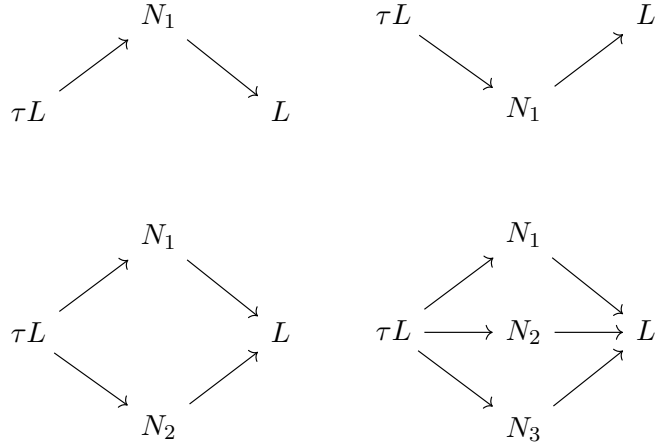

\begin{lemma}[\cite{assem_elements_2006} Corollary VIII.2.3] \label{lem-postmod}
    Assume that $\mathcal{A} = \k Q$ is the path algebra of a finite acyclic quiver $Q$ of affine Dynkin type. 
    Then $\Gamma(\mod \mathcal{A} )$ contains a unique preprojective component.
    Specifically, $\mathcal{P(A)} \cong \mathbb{N} Q^{op}$ where $Q^{op}$ is the opposite quiver of $Q$ obtained by reversing all arrows in $Q$.
\end{lemma}

\section{Equivalence of Categories}

\begin{theorem}
    Let $\TT$ be a triangulation of a twice-punctured $n$-gon.
    There is an equivalence of categories 
    $$\varphi : \mathcal{P}(\TT) \to \mathcal{P}(\k Q^\TT)$$
    such that 
    \begin{enumerate}
        \item $\varphi$ maps projective tagged edges to projective indecomposable modules and respects the labeling
        \item $\varphi \circ \rho = \tau \circ \varphi$
    \end{enumerate}
\end{theorem}

\begin{proof}
    Fix an admissible triangulation $\TT$ of a twice-punctured $n$-gon. 
    This gives rise to an acyclic quiver $Q^\TT$ of type $\widetilde{D}_{n+2}$ with each tagged edge $i \in \TT$ corresponding to a vertex $i \in Q^\TT_0$. 
    Let $\k Q^\TT$ be the path algebra of this quiver.
    
    The projective tagged edges are the edges such that $\rho^1(\gamma, \kappa) \in \TT$.
    Specifically, there is precisely one projective tagged edge corresponding to each of the $n+3$ elements of the triangulation $\TT$.
    Similarly, there is exactly one projective module for each of the $n+3$ vertices of the quiver $Q^\TT$ of type $\widetilde{D}_{n+2}$. 
    If a tagged edge $(\gamma,\kappa) \in \TT$ corresponds to vertex $i$ in $Q^\TT$, then $\varphi (\rho^{-1}(\gamma, \kappa)) \mapsto P(i)$. 
    Therefore, $\varphi$ maps projective tagged edges to projective indecomposable modules.

    Since the two categories have the same structure (Lemmas~\ref{lem-postedge} and~\ref{lem-postmod}) and they have the same basic elements, the result follows.
\end{proof}

The equivalence of categories identifies a homotopy class of a preprojective tagged edge with an isomorphism class of an indecomposable preprojective module $M = M(\gamma,\kappa)$.
By observing the intersections of a curve with the triangulation and comparing this with the dimension vector of the module it represents, the following result is immediate.

\begin{corollary}\label{cor-int-mod}
    Let $\TT$ be a triangulation of a twice-punctured $n$-gon. Then a preprojective tagged edge $(\gamma, \kappa) \in PPTE(\TT)$ corresponds to the isomorphism class of the indecomposable module with dimension vector $$\udim M(\gamma, \kappa) = (\Int((\gamma,\kappa), i ))_{i \in \TT}.$$
\end{corollary}

\section{The Intersection-Dimension Formula}

\begin{figure}
    \centering
    \includegraphics[scale = 0.5]{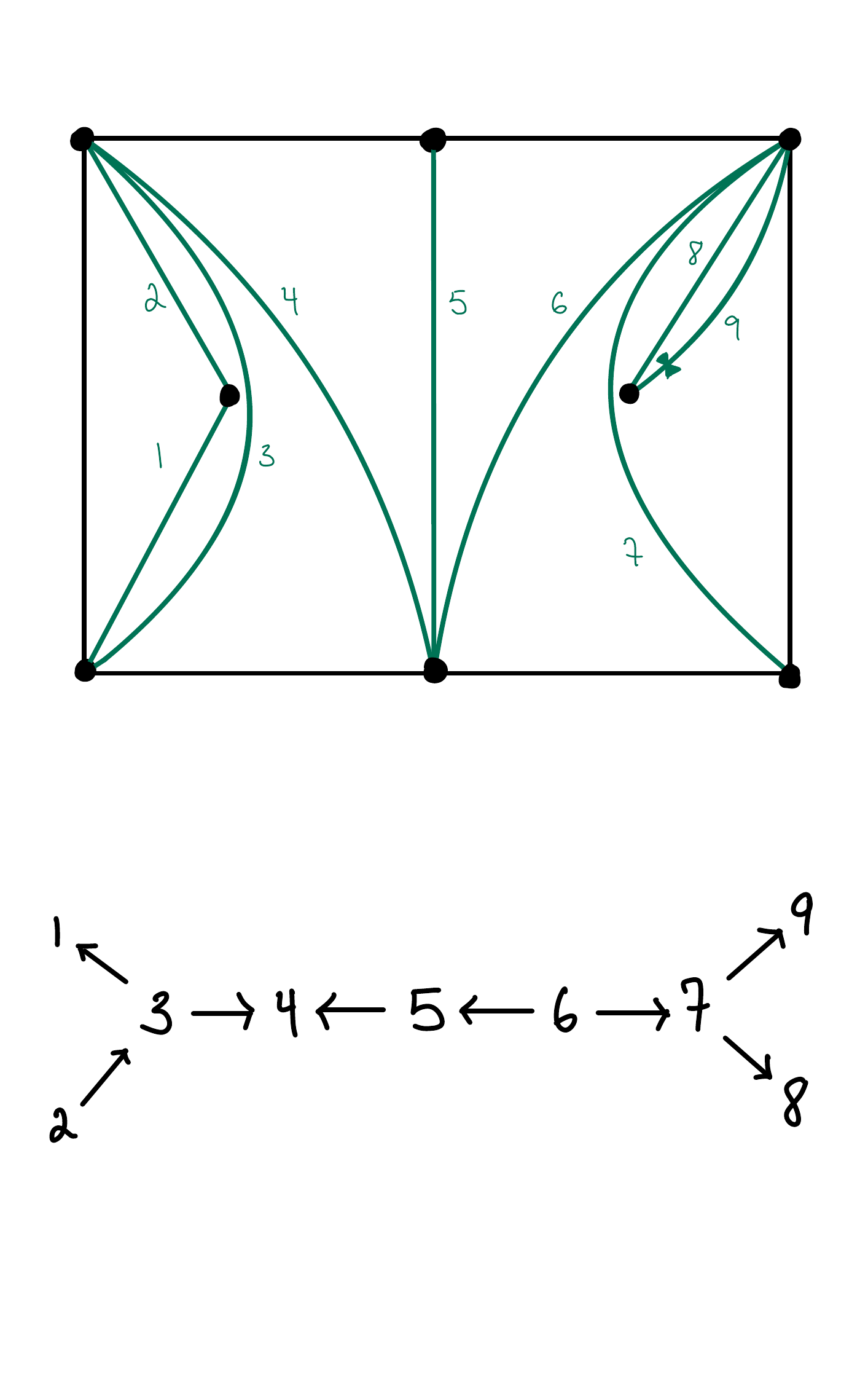}
    \caption{A triangulation of type $\widetilde{D}_8$ with along with $Q^\TT$.}
    \label{fig:example-triangulation}
\end{figure}

The equivalence of categories along with Corollary~\ref{cor-int-mod} suggests there might be a way of interpreting the homological data in the preprojective module category in terms of the geometric data in the category of preprojective tagged edges since, upon inspection, both follow similar patterns.
For the sake of the completeness of this paper, some definitions which are used in the Auslander-Reiten formulas are given before the proof of the intersection-dimension formula. 
For a more in-depth treatment of these topics, see \cite{assem_elements_2006,auslander_representation_1995,schiffler_quiver_2014}.

\begin{definition}
    The \textbf{duality} $$D = \Hom_\k (-,\k): \operatorname{rep} Q \to \operatorname{rep} Q^{op}$$ is the contravariant functor defined as follows:
    \begin{itemize}
        \item For representations $M = (M_i, \varphi_\alpha)$, we have $$DM = (DM_i, D\varphi_{\alpha^{op}})_{i\in Q_0, \alpha \in Q_1},$$ where $DM_i$ is the dual vector space and if $\alpha$ is an arrow in $Q$ then $D\varphi_{\alpha^{op}}$ is the pullback of $\varphi_\alpha$ \begin{align*}
            D\varphi_{\alpha^{op}}: DM_{t(\alpha)} &\to DM_{s(\alpha)} \\
                                    u &\mapsto u \circ \varphi_\alpha.
        \end{align*}
        \item For homomorphisms $f:M \to N$ in $\operatorname{rep} Q$, we have $Df : DN \to DM$ in $\operatorname{rep} Q^{op}$ defined by $Df(u) = u \circ f$.
    \end{itemize}
\end{definition}

\begin{definition} \label{linehome}
    Let $P(M,N)$ be the set of all homomorphisms $f \in \Hom(M,N)$ such that $f$ factors through a projective $\mathcal{A}$-module, and define $$\underline\Hom(M,N) = \Hom(M,N)/P(M,N).$$
    Dually, let $I(M,N)$ be the set of all homomorphisms $f \in \Hom(M,N)$ such that $f$ factors through an injective $\mathcal{A}$-module, and define $$\overline\Hom(M,N) = \Hom(M,N)/I(M,N).$$
\end{definition}

\begin{theorem}[Auslander-Reiten formulas]
    Let $M,N$ be $\mathcal{A}$-modules. 
    Then there are isomorphisms $$\Ext^1(M,N) \cong D \underline{\Hom} (\tau^{-1}N, M) \cong D \overline{\Hom} (N, \tau M).$$
\end{theorem}

\begin{corollary} \label{ext-hom}
    Let $M,N$ be $\mathcal{A}$-modules. Then $$\dim_\k \Ext^1(M,N) = \dim_\k \Hom (N,\tau M) = \dim_\k \Hom(\tau^{-1}N, M).$$
\end{corollary}

The proof of the intersection-dimension formula requires knowledge about the relative position of two objects in the category of preprojective tagged edges (or modules).
This necessitates the following definition.

\begin{figure}
    \centering
    \includegraphics[scale = 0.7]{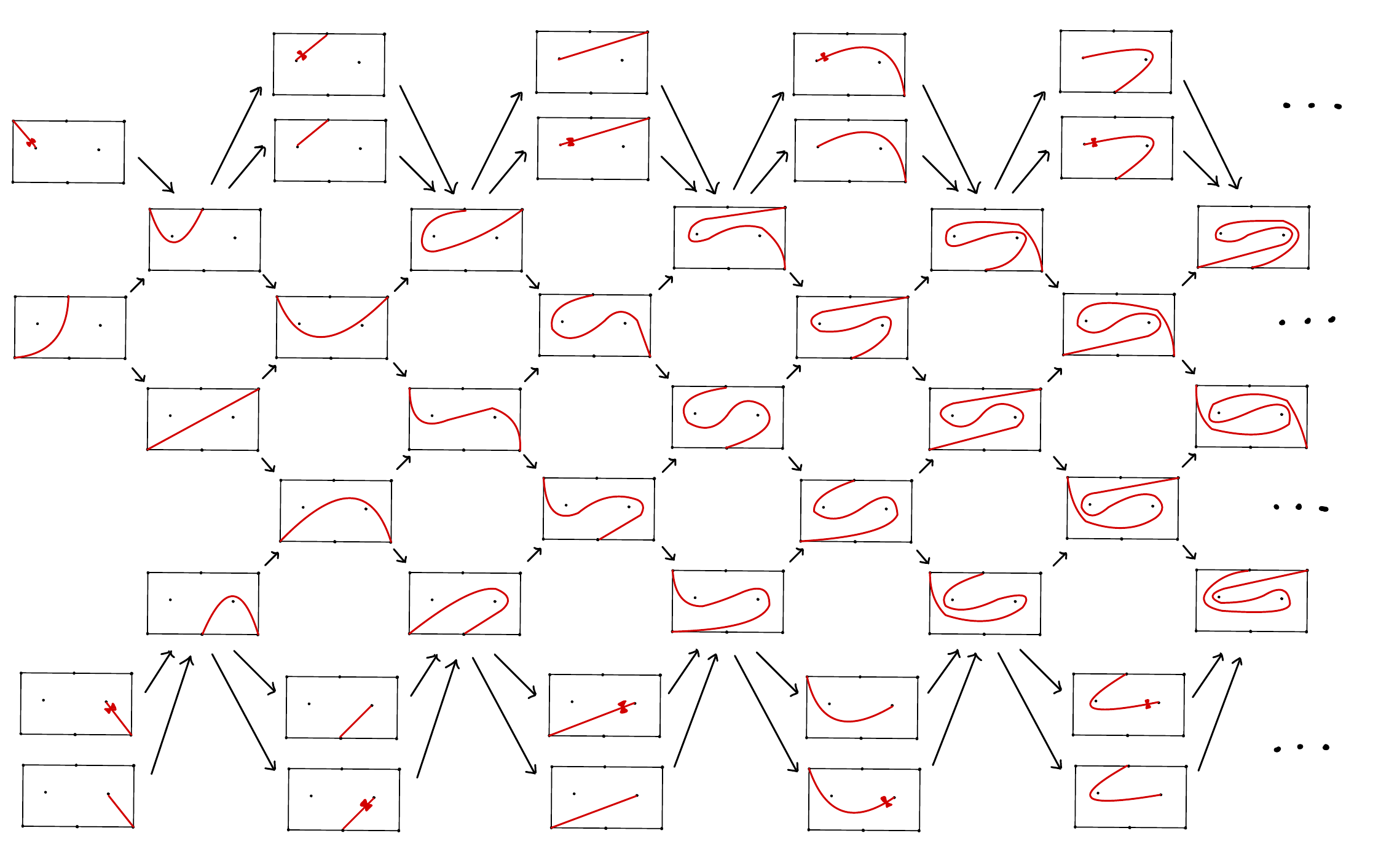}
    \caption{The beginning of $\mathcal{P(\TT)}$ for the triangulation in Figure~\ref{fig:example-triangulation}.}
    \label{fig:example-curves}
\end{figure}

\begin{figure}
    \centering
    \includegraphics[scale = 0.7]{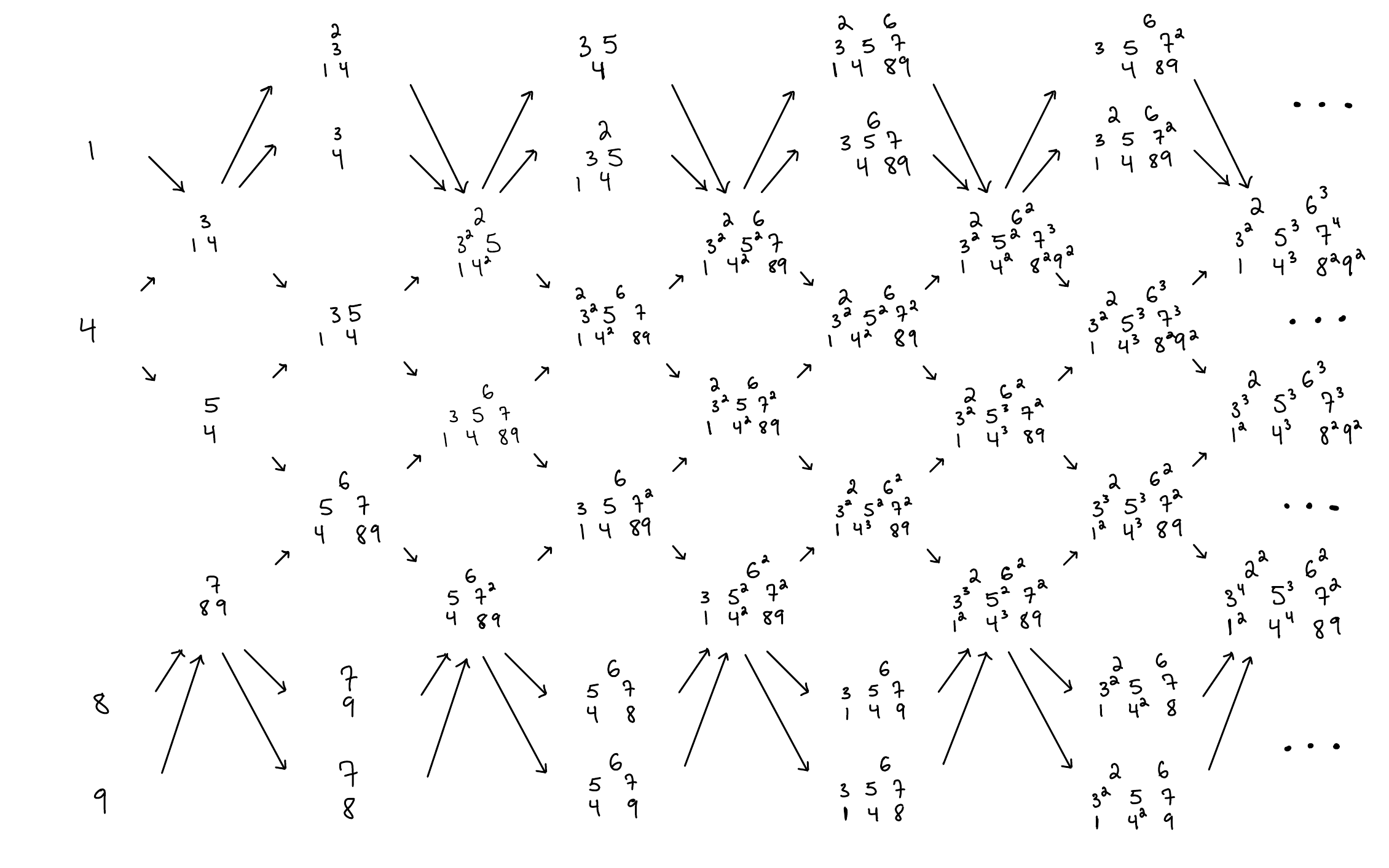}
    \caption{The beginning of $\mathcal{P}(\k Q^\TT)$ for the triangulation in Figure~\ref{fig:example-triangulation}. Dimension vectors in the style of \cite{schiffler_quiver_2014} are given for each isoclass of indecomposable module.}
    \label{fig:example-modules}
\end{figure}

\begin{definition}\label{def-relativecoord}
    Fix a module $M_i = M(\gamma_i, \kappa_i) \in \mathcal{P}(\k Q^\TT)$. 
    Then another module $N \in \mathcal{P}(\k Q^\TT)$ has \textbf{coordinates relative to $M_i$} of $(a,b)^\kappa_{M_i}$ if $N$ can be reached from $M_i$ using $a \in \mathbb{Z}$ applications of $\tau^{-1} = \rho^{-1}$, $b \in \mathbb{Z}$ moves along a (directed) diagonal, and tagging the resulting edge using the map $\kappa \in \{0,1,\emptyset\}$.
\end{definition}

\begin{remark}
    Recall that $\tau^{-1} = \rho^{-1}$ moves you horizontally to the right through $\mathcal{P}(\k Q^\TT)$, so $a \geq 0$ if and only if $N$ lies on or to the right of the directed diagonals emanating from $M_i$.
    Also, $b < 0$ if you move down along a diagonal, $b > 0$ if you move up along a diagonal, and $b = 0$ if no diagonal moves are required.
    For example, in Figure~\ref{fig:example-ext-dims}, the slanted rectangle of 1's has vertices with relative coordinates $(1,0)^\emptyset_M, (1,2)^\emptyset_M, (1,-5)^\emptyset_M$ and $(3,-3)^\emptyset_M$.
    It is worth mentioning that there are some modules with the same relative coordinates but different tagging; in Figure~\ref{fig:example-ext-dims}, these have relative coordinates $(a,3)^\kappa_M$ and $(a,-6)^\kappa_M$ for all $a \geq 0$ and $\kappa = 0,1$. 
\end{remark}

\begin{figure}
\[\begin{tikzcd}[scale cd = 0.6, sep = tiny]
     \cdots &  & 0 &   & 0 &   & 1 &   & 1 &   & 1 &   & 1 &   & 1 &   & 1 &   & 2 &   & 2 &   & 2 &   & 3 &   & \cdots \\
            &  & 0 &   & 0 &   & 1 &   & 1 &   & 1 &   & 1 &   & 1 &   & 1 &   & 2 &   & 2 &   & 2 &   & 3 &   & \\
     \cdots &0 &   & 0 &   & 1\arrow{rd} &   & 2 &   & 2 &   & 2 &   & 2 &   & 2 &   & 3\arrow{rd} &   & 4 &   & 4 &   & 5\arrow{rd} &   & 6 & \cdots \\
            &  & 0 &   & 1\arrow{ru} &   & 1\arrow{rd} &   & 2 &   & 2 &   & 2 &   & 2 &   & 3\arrow{ru} &   & 3\arrow{rd} &   & 4 &   & 5\arrow{ru} &   & 5\arrow{rd} &   & \\
     \cdots &M &   & 1\arrow{ru}\arrow{rd} &   & 1 &   & 1\arrow{rd} &   & 2 &   & 2 &   & 2 &   & 3\arrow{ru} &   & 3 &   & 3 &   & 5\arrow{ru}\arrow{rd} &   & 5 &   & 5 & \cdots \\
            &  & 0 &   & 1\arrow{rd} &   & 1 &   & 1\arrow{rd} &   & 2 &   & 2 &   & 3\arrow{ru} &   & 3 &   & 3\arrow{ru} &   & 4 &   & 5\arrow{rd} &   & 5 &   & \\
     \cdots &0 &   & 0 &   & 1\arrow{rd} &   & 1 &   & 1\arrow{rd} &   & 2 &   & 3\arrow{ru} &   & 3 &   & 3\arrow{ru} &   & 4 &   & 4 &   & 5\arrow{rd} &   & 5 & \cdots \\
            &  & 0 &   & 0 &   & 1\arrow{rd} &   & 1 &   & 1 &   & 3\arrow{ru}\arrow{rd} &   & 3 &   & 3\arrow{ru} &   & 4 &   & 4 &   & 4 &   & 5\arrow{rd} &   & \\
     \cdots &0 &   & 0 &   & 0 &   & 1\arrow{rd} &   & 1\arrow{ru} &   & 2 &   & 3\arrow{rd} &   & 3\arrow{ru} &   & 4 &   & 4 &   & 4 &   & 4 &   & 5 & \cdots \\
            &  & 0 &   & 0 &   & 0 &   & 1\arrow{ru} &   & 2 &   & 2 &   & 3\arrow{ru} &   & 4 &   & 4 &   & 4 &   & 4 &   & 4 &   & \\
            &0 &   & 0 &   & 0 &   & 0 &   & 1 &   & 1 &   & 1 &   & 2 &   & 2 &   & 2 &   & 2 &   & 2 &   & 2 & \\
    \cdots  &0 &   & 0 &   & 0 &   & 0 &   & 1 &   & 1 &   & 1 &   & 2 &   & 2 &   & 2 &   & 2 &   & 2 &   & 2 & \cdots
    \end{tikzcd}\]
    \caption{Some values of $\dim_\k\Ext(-,M) = \dim_\k\Hom(\tau^{-1}M,-)$ for a module in $\mathcal{P}(\k Q^\TT)$ of type $\widetilde{D}_{11}$ with selected arrows included to show the pattern.}
    \label{fig:example-ext-dims}
\end{figure}
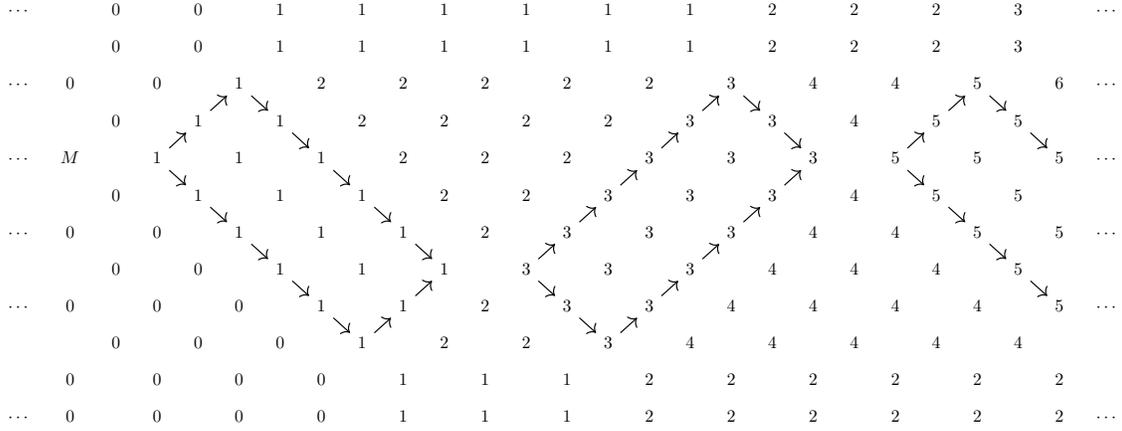

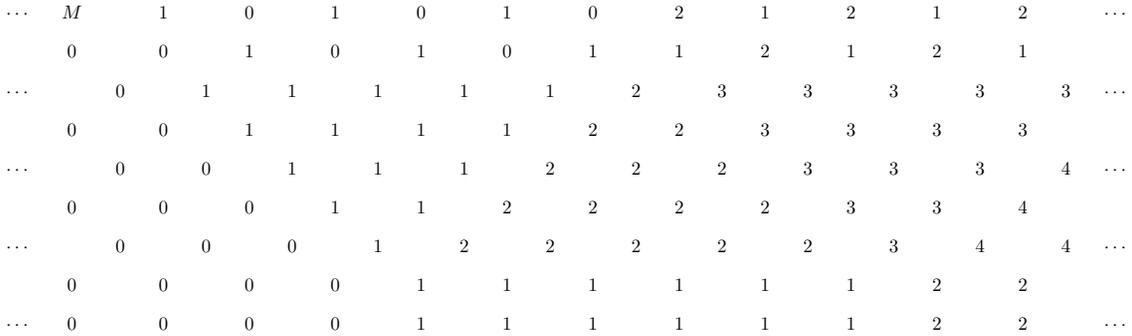
\begin{figure}
\[\begin{tikzcd}[scale cd = 0.65, sep = tiny]
     \cdots & M &   & 1 &   & 0 &   & 1 &   & 0 &   & 1 &   & 0 &   & 2 &   & 1 &   & 2 &   & 1 &   & 2 &   & \cdots \\
            & 0 &   & 0 &   & 1 &   & 0 &   & 1 &   & 0 &   & 1 &   & 1 &   & 2 &   & 1 &   & 2 &   & 1 &   & \\
     \cdots &   & 0 &   & 1 &   & 1 &   & 1 &   & 1 &   & 1 &   & 2 &   & 3 &   & 3 &   & 3 &   & 3 &   & 3 & \cdots \\
            & 0 &   & 0 &   & 1 &   & 1 &   & 1 &   & 1 &   & 2 &   & 2 &   & 3 &   & 3 &   & 3 &   & 3 &   & \\
     \cdots &   & 0 &   & 0 &   & 1 &   & 1 &   & 1 &   & 2 &   & 2 &   & 2 &   & 3 &   & 3 &   & 3 &   & 4 & \cdots \\
            & 0 &   & 0 &   & 0 &   & 1 &   & 1 &   & 2 &   & 2 &   & 2 &   & 2 &   & 3 &   & 3 &   & 4 &   & \\
     \cdots &   & 0 &   & 0 &   & 0 &   & 1 &   & 2 &   & 2 &   & 2 &   & 2 &   & 2 &   & 3 &   & 4 &   & 4 & \cdots \\
            & 0 &   & 0 &   & 0 &   & 0 &   & 1 &   & 1 &   & 1 &   & 1 &   & 1 &   & 1 &   & 2 &   & 2 &   &  \\
    \cdots  & 0 &   & 0 &   & 0 &   & 0 &   & 1 &   & 1 &   & 1 &   & 1 &   & 1 &   & 1 &   & 2 &   & 2 &   & \cdots
    \end{tikzcd}\]
    \caption{Some values of $\dim_\k\Ext(-,M) = \dim_\k\Hom(\tau^{-1}M,-)$ for a module in $\mathcal{P}(\k Q^\TT)$ of type $\widetilde{D}_{8}$.}
    \label{fig:example-ext-dims-edge}
\end{figure}

\begin{theorem}
    Let $(\gamma_1, \kappa_1), (\gamma_2, \kappa_2) \in \mathcal{P}(\TT)$ where $\TT$ is a triangualtion of a twice-punctured $n$-gon $\surf$ and let $M(\gamma_1, \kappa_1) = M_1, M(\gamma_2, \kappa_2) = M_2 \in \mathcal{P}(\k Q^\TT)$ be the corresponding modules under the equivalence $\varphi$.
    Then $$\Int( (\gamma_1,\kappa_1), (\gamma_2,\kappa_2) ) = \dim_\k \Hom (M_2, \tau M_1) + \dim_\k \Hom (M_1, \tau M_2).$$
\end{theorem}

\begin{proof}
    If $M_1 = M_2$, then $(\gamma_1, \kappa_1) \sim (\gamma_2, \kappa_2)$. 
    Therefore, $$\Int( (\gamma_1,\kappa_1), (\gamma_1,\kappa_1) ) = 0 = 2\dim_\k \Hom (M_1, \tau M_1).$$
    We can be sure that no preprojective tagged edges have self-intersections since these are preserved under the action of $\rho$; on the other hand, there are no homomorphisms from a module to its Auslander-Reiten translation.

    Assume $M_1 \neq M_2$. 
    If $level(M_1) = n = level(M_2)$, both $\Hom$ spaces are zero since there will be no sequences of irreducible morphisms from one module to the Auslander-Reiten translate of the other (in other words, no directed paths from one module to the other in $\mathcal{P}(\k Q^\TT)$).
    On the other hand, the curves $(\gamma_1, \kappa_1)$ and $(\gamma_2, \kappa_2)$ representing these modules must also have $level(\gamma_1, \kappa_1) = n = level(\gamma_2, \kappa_2)$ due to the equivalence of categories.
    Therefore, both curves are the same number of tagged rotations $\rho$ away from elements in the triangulation. 
    Since $\rho$ preserves adjacency and all elements of the triangulation are non-crossing by definition, it must be the case that $\Int( (\gamma_1,\kappa_1), (\gamma_2,\kappa_2) ) = 0$.

    If, without loss of generality, $level(M_1) = n_1 > n_2 = level(M_2)$, then $\Hom(M_1, \tau M_2) = 0$. 
    Therefore, all of the homological data will come from $\Hom(M_2, \tau M_1)$ as long as the relative coordinates for $\tau M_1$ are $(a,b)^\kappa_{M_2}$ with $a \geq 0$.
    If $a < 0$, then $\Hom(M_2, \tau M_1) = 0$ as well.
    We have a few cases to consider since each of $(\gamma_1, \kappa_1)$ and $(\gamma_2, \kappa_2)$ can either be a curve with both endpoints in $\MM$ or a curve with one endpoint in $\MM$ and the other in $\PP$.
    Even still, for the curves ending in a puncture, there are two punctures to end in.
    The easiest way to proceed is to fix $M_2 = M(\gamma_2, \kappa_2)$ and consider the possible relative position of $\tau M_1 = \tau M(\gamma_1, \kappa_1)$.
    
    There are four cases for the relative position of two modules in $\mathcal{P}(\k Q^\TT)$: both modules are in the ``middle part'' or the ``type $A_n$ part'' (for example, the module $M$ in Figure~\ref{fig:example-ext-dims} is in the type $A_n$ part), one module is in the type $A_n$ part and the other is in the ``type $D_n$ part'' (the modules $M$ in Figure~\ref{fig:example-ext-dims-edge} is in the type $D_n$ part), both modules are in the same type $D_n$ part, or both modules are in different type $D_n$ parts. 
    Since some of these cases can be combined in the proof, we have three cases to consider.

    \underline{Case 1:} $M_2$ is in the type $A_n$ part of $\mathcal{P}(\k Q^\TT)$.

    In this case, $\gamma_2$ has both endpoints in $\MM$ so any intersections between the curves will be normal intersections. 
    We know that for some $N$ in $\mathcal{P}(\k Q^\TT)$, $\Hom(N, -)$ can be calculated by drawing ``maximal slanted rectangles'' or ``maximal hammocks'' (cf. the arrows included in Figure~\ref{fig:example-ext-dims}).
    By Corollary~\ref{ext-hom} and since we have fixed $M_2$, it is slightly more natural to work with $\Hom(\tau^{-1}M_2, M_1)$ when verifying the statement of the theorem.

    Figure~\ref{fig:example-ext-dims} shows us the pattern that $\dim_\k \Hom(\tau^{-1}M, -)$ follows.
    What is happening in terms of preprojective tagged edges is shown in Figure~\ref{fig:example-curves} for a specific example.
    We can see that moving along a diagonal in $\mathcal{P}(\TT)$ fixes one endpoint of a curve and acts by inverse elementary moves on the other end.
    If $M_1$ is also in the type $A_n$ part, each inverse elementary move will either keep the intersection number the same as it was or increase it by 1 in precisely the same way that the maximal slanted rectangles capture the increases in the homological data.
    
    To describe this pattern of maximal slanted rectangles for a general category of type $\widetilde{D}_n$ requires a combination of coordinates and relative coordinates.
    Since $M_2$ is in the type $A_n$ part, it will have coordinates $(n_2,i)$ where $3 \leq i \leq (n-1)$ and $\tau^{-1}M_2$ will have coordinates $(n_2+1,i)$.
    For the sake of brevity, assume that $i \leq \frac{n+1}{2}$ so that $M_2$ is in the ``top half'' of the type $A_n$ part (like $M$ in Figure~\ref{fig:example-ext-dims}).
    All of the slanted rectangles have a semi-perimeter of $n-4$, initial corners at $((n-2)\ell, 0)^\emptyset_{\tau^{-1}M_2}$ or $(i-2 + (n-2)\ell, (n-4)-(i-3))^\emptyset_{\tau^{-1}M_2}$ for $\ell \in \mathbb{N}_0$, and terminal corners at $(i-3 + (n-2)\ell, (n-4)-(i-3))^\emptyset_{\tau^{-1}M_2}$ or $((n-3)\ell, 0)^\emptyset_{\tau^{-1}M_2}$ (respectively) for $\ell \in \mathbb{N}$. The modules that live on or inside of the maximal slanted rectangles have $\Int( (\gamma_1,\kappa_1), (\gamma_2,\kappa_2) ) = \dim_\k \Hom(\tau^{-1}M_2,M_1) = 2s-1$ where $s$ indicates which maximal slanted rectangle $M_1$ belongs to counting from left to right.
    The modules $M_1$ between the $s_i$ and $s_{i+1}$ maximal slanted rectangles have $\Int( (\gamma_1,\kappa_1), (\gamma_2,\kappa_2) ) = \dim_\k \Hom(\tau^{-1}M_2,M_1) = 2s_i$.
    
    If $M_1$ is in the type $D_n$ part, the situation is slightly more subtle. 
    The move from the type $A_n$ part to the type $D_n$ part happens when an inverse elementary move will cause a curve to cut out a once-punctured monogon, call it $\lambda$.
    In this case, we replace $\lambda$ with the two preprojective tagged edges for which $\lambda$ is their completion.
    In order to see what is going on in terms of the intersection data, it is easiest to follow a diagonal into the type $D_n$ part of $\mathcal{P}(\TT)$ and then move horizontally using $\rho^{-1}$ until you reach $(\gamma_1, \kappa_1)$.
    Comparing the intersection pattern in Figure~\ref{fig:example-curves} with the $\Hom$ dimensions in the type $D_n$ part of Figure~\ref{fig:example-ext-dims}, we can verify that these two share the same pattern.
    
    Namely, if $M_2$ has coordinates $(n_2,i)$ where $3 \leq i \leq (n-1)$ and $\tau^{-1}M_2$ will has coordinates $(n_2+1,i)$, then the pattern of $\Int( (\gamma_1,\kappa_1), - ) = \dim_\k \Hom(\tau^{-1}M_2,-)$ in the upper type $D_n$ part beginning at $(0, i-2)^\kappa_{\tau^{-1}M_2}$ and moving to the right is $(n-i)$ 1's, $(i-2)$ 2's, $(n-i)$ 3's, $(i-2)$ 4's, and so on.
    In the lower type $D_n$ part, the pattern of $\Int( (\gamma_1,\kappa_1), - ) = \dim_\k \Hom(\tau^{-1}M_2,-)$ beginning at $(0, n-i)^\kappa_{\tau^{-1}M_2}$ and moving to the right is $(i-2)$ 1's, $(n-i)$ 2's, $(i-2)$ 3's, $(n-i)$ 4's, and so on.
    The easier way to describe this pattern is by keeping track of the corners of the maximal slanted rectangles which touch the type $D_n$ parts.
    If $M_1$ is in either type $D_n$ part, it will be between two corner points of neighboring maximal slanted rectangles, call the rectangles $s_i$ and $s_{i+1}$. 
    Then $\Int( (\gamma_1,\kappa_1), (\gamma_2,\kappa_2) ) = \dim_\k \Hom(\tau^{-1}M_2,M_1) = i$. 

    Therefore, $$\Int( (\gamma_1,\kappa_1), (\gamma_2,\kappa_2) ) = \dim_\k \Hom (\tau^{-1}M_2, M_1) = \dim_\k \Hom (M_2, \tau M_1).$$

    \underline{Case 2:} $M_2$ is in the type $D_n$ part and $M_1$ is in the type $A_n$ part of $\mathcal{P}(\k Q^\TT)$.

    Again, all intersections will be normal intersections since $\gamma_1$ has both endpoints in $\MM$.
    This case follows a similar argument as the previous subcase but the roles of $M_1$ and $M_2$ are reversed.
    Assume that $\tau^{-1}M_2$ is in the top type $D_n$ part with coordinates $(n_2+1,1)$ or $(n_2+1,2)$ so that we are in a case similar to that of Figure~\ref{fig:example-ext-dims-edge}.
    In this case, instead of maximal slanted rectangles in the type $A_n$ part, we get ``maximal triangles''.
    The odd-numbered triangles (first, third, fifth, etc.) have vertices $((n-2)\ell, -1)^\emptyset_{\tau^{-1}M_2}, (n-4 + (n-2)\ell, -1)^\emptyset_{\tau^{-1}M_2},$ and $((n-2)\ell, -(n-3))^\emptyset_{\tau^{-1}M_2}$ for $\ell \in \mathbb{N}_0$.
    On the other hand, the even-numbered triangles have vertices $(1 + (n-2)\ell, -(n-3))^\emptyset_{\tau^{-1}M_2}, ((n-2)(\ell+1) - 1, -(n-3))^\emptyset_{\tau^{-1}M_2},$ and $(n-3 + (n-2)\ell, -1)^\emptyset_{\tau^{-1}M_2}$ for $\ell \in \mathbb{N}_0$.

    Again, $$\Int( (\gamma_1,\kappa_1), (\gamma_2,\kappa_2) ) = \dim_\k \Hom (\tau^{-1}M_2, M_1) = \dim_\k \Hom (M_2, \tau M_1).$$

    \underline{Case 3:} $M_2$ and $M_1$ are in the type $D_n$ parts of $\mathcal{P}(\k Q^\TT)$.

    If $M_2$ and $M_1$ are not in the same type $D_n$ part, then $(\gamma_2,\kappa_2)$ and $(\gamma_1,\kappa_1)$ will have endpoints in different punctures.
    Therefore, all intersections will be normal intersections.
    Figure~\ref{fig:example-ext-dims-edge} shows $\dim_\k \Hom(\tau^{-1}M_2, M_1)$ when $M_2$ is in the type $D_n$ part of $\mathcal{P}(\k Q^\TT)$.
    Comparing this with Figure~\ref{fig:example-curves} shows the intersection patterns for two curves in opposite type $D_n$ parts of $\mathcal{P}(\TT)$.
    The $(n-2)$ periodic behavior in the category of preprojective tagged edges is explained by the number of marked points on the boundary.

    Say $M_2$ has coordinates $(n_2,i)$ where $i = 1,2$ and $\tau^{-1}M_2$ has coordinates $(n_2+1,i)$.
    Since $M_1$ is in the other type $D_n$ part, the pattern begins at $(0,n-2)^\kappa_{\tau^{-1}M_2}$.
    If $M_1$ has relative coordinates $(a,n-2)^\kappa_{\tau^{-1}M_2}$ where $a = (n-2)\ell + r$ and $r < (n-2)$, $\Int( (\gamma_1,\kappa_1), (\gamma_2,\kappa_2) ) = \dim_\k \Hom (\tau^{-1}M_2, M_1) = \ell$. 

    If $M_2$ and $M_1$ are in the same type $D_n$ part, then $(\gamma_2,\kappa_2)$ and $(\gamma_1,\kappa_1)$ will have endpoints in the same puncture.
    Therefore, intersections will involve normal as well as punctured intersections.
    The alternating pattern seen in the upper type $D_n$ part of Figure~\ref{fig:example-ext-dims-edge} is explained by the tagged rotation alternating the tagging on the preprojective tagged edges. 
    The change from the $0-1$ pattern to the $1-2$ pattern occurs when a normal intersection is introduced.

    Say $M_2$ has coordinates $(n_2,i)$ where $i = 1,2$, $\tau^{-1}M_2$ has coordinates $(n_2+1,i)$, and $\tau^{-1}M_2$ has tagging $\kappa$.
    Since $M_1$ is in the same type $D_n$ part, the pattern begins at $(0,0)^\kappa_{\tau^{-1}M_2}$.
    If $M_1$ has relative coordinates $(a,0)^\kappa_{\tau^{-1}M_2}$ where $a = (n-2)\ell + r$ and $r < (n-2)$, $\Int( (\gamma_1,\kappa_1), (\gamma_2,\kappa_2) ) = \dim_\k \Hom (\tau^{-1}M_2, M_1) = \ell - 1$.
    If $M_1$ has relative coordinates $(a,0)^{1-\kappa}_{\tau^{-1}M_2}$ where $a = (n-2)\ell + r$ and $r < (n-2)$, $\Int( (\gamma_1,\kappa_1), (\gamma_2,\kappa_2) ) = \dim_\k \Hom (\tau^{-1}M_2, M_1) = \ell$.

    In any case, we have $$\Int( (\gamma_1,\kappa_1), (\gamma_2,\kappa_2) ) = \dim_\k \Hom (M_2, \tau M_1) + \dim_\k \Hom (M_1, \tau M_2).$$
\end{proof}

\begin{corollary}[Theorem A]
    Let $\surf$ be a triangulated, twice-punctured marked surface whose triangulation $\TT$ corresponds to an acyclic quiver $Q^\TT$ of type $\widetilde{D}_n$. 
    Then given any two preprojective tagged edges $(\gamma_1, \kappa_1)$ and $(\gamma_2, \kappa_2)$ (not necessarily distinct), 
    $$\Int((\gamma_1, \kappa_1), (\gamma_2, \kappa_2)) = \dim_\k \Ext^1 (M_1, M_2) + \dim_\k \Ext^1 (M_2, M_1)$$ where $M_i = M(\gamma_i,\kappa_i)$ and $\Int$ is the intersection number between two admissible tagged edges.
\end{corollary}

\begin{proof}
    By Corollary~\ref{ext-hom}, $\dim_\k \Ext^1 (M_1, M_2) = \dim_\k \Hom (M_2, \tau M_1).$
    The result is immediate.
\end{proof}

\bibliography{references}
\bibliographystyle{amsplain}


\end{document}